\documentclass[11pt,reqno,twoside]{amsart}
\usepackage{graphicx}
\usepackage{amssymb}
\usepackage{amsthm}
\usepackage{epstopdf}
\usepackage[asymmetric,top=2.5cm,bottom=2.8cm,left=2.38cm,right=2.38cm]{geometry}
\usepackage{hyperref}

\usepackage{booktabs} 
\usepackage{array} 
\usepackage{paralist} 
\usepackage{verbatim} 
\usepackage{subfig} 
\usepackage{tabularx}
\usepackage{amsmath,amsfonts,amsthm,mathrsfs,amssymb,cite}
\usepackage[usenames]{color}
\usepackage{bm}
\usepackage{color}
\usepackage{extarrows}
\usepackage{xcolor}
\hypersetup{
	colorlinks,
	linkcolor={blue!80!black},
	urlcolor={blue!80!black},
	citecolor={blue!80!black}
}

\newtheorem{thm}{Theorem}[section]
\newtheorem{cor}[thm]{Corollary}
\newtheorem{lem}{Lemma}[section]
\newtheorem{prop}{Proposition}[section]
\theoremstyle{definition}
\newtheorem{defn}{Definition}[section]
\theoremstyle{remark}

\newtheorem{rem}{Remark}[section]
\numberwithin{equation}{section}

\newcommand{\bmf}[1]{{\mathbf{#1}}}
\newcommand\rmi{{\mathrm{i}}}

\newcommand{\rmd}{\mathrm{d}}

\allowdisplaybreaks[4]

\title[On quasi-Minnaert resonances in elasticity]{On quasi-Minnaert resonances in elasticity and their applications to stress concentrations}

\author{Huaian Diao}
\address{School of Mathematics, Jilin University and Key Laboratory of Symbolic Computation and Knowledge Engineering of Ministry of Education, Changchun, Jilin, China.}
\email{diao@jlu.edu.cn, hadiao@gmail.com}

\author{Ruixiang Tang}
\address{School of Mathematics, Jilin University, Changchun 130012, China.}
\email{tangrx97@gmail.com; tangrx23@mails.jlu.edu.cn}

\author{Hongyu Liu}
\address{Department of Mathematics, City University of Hong Kong, Kowloon, Hong Kong SAR, China.}
\email{hongyu.liuip@gmail.com; hongyliu@cityu.edu.hk}

\date{} 
\begin{document}
	
	\maketitle
	\begin{abstract}

		This paper unveils and investigates a novel quasi-Minnaert resonance for an elastic hard inclusion embedded in a soft homogeneous medium in the sub-wavelength regime. The quasi-Minnaert resonance consists of boundary localization and surface resonance for the generated internal total and external scattered wave fields associated with the hard inclusion. It possesses similar quantitative behaviours as those for the classical Minnaert resonance due to high-contrast material structures, but occurs for a continuous spectrum of frequencies instead of certain discrete Minnaert resonant frequencies. We present a comprehensive analysis to uncover the physical origin and the mechanism of this new physical phenomenon. It is shown that the delicate high-contrast material structures and the properly tailored incident waves which are coupled together in a subtle manner play a crucial role in ensuring such phenomena. The stress concentration phenomena in both the internal total field and the scattered field components are also rigorously established. The analysis in this paper is deeply rooted in layer potential theory and intricate asymptotic analysis. We believe that our findings can have a significant impact on  the theory of composite materials and metamaterials.

	\medskip
		
		
	\end{abstract}
	

	\section{Introduction}

	In this paper, we investigate the scattering problem associated with a hard elastic inclusion embedded within a soft elastic material in the sub-wavelength regime. Such physical models are of fundamental importance in the design and construction of metamaterials with negative material parameters, which have garnered significant attention in recent years. Notable advancements in this field include the realization of negative density and refractive index in acoustics \cite{ADM, AMRZ}, negative permittivity and permeability in electromagnetism \cite{ACK, ARY, AR, CKK}, and negative Lam\'e parameters in elasticity \cite{LIUCHAN,LIUZHANG}. These developments have enabled a wide range of industrial applications, including antennas \cite{ETSG}, absorbers \cite{LIVA}, invisibility cloaking \cite{ACK}, super-lenses \cite{FLAZ, PJB}, and super-resolution imaging \cite{ADM}.

	The sub-wavelength regime corresponds to the case where the size of the inclusion is significantly smaller than the operating wavelength of the incident wave. In acoustics, for instance, the immersion of a bubble in a liquid creates a high density contrast between the bubble and the surrounding medium, leading to a distinct low-frequency resonance known as the Minnaert frequency. At this frequency, bubbles act as acoustic resonators \cite{Min33}, playing a pivotal role in the design of novel materials, including phononic crystals. The Minnaert resonance in acoustic settings was first rigorously studied in \cite{AFGL}, where the relationship between Minnaert resonance frequencies and high contrast parameters was established. Analogous developments in elastic scattering problems have been explored in \cite{lizou}. Furthermore, it has been demonstrated in \cite{ACC,DGS} that the Minnaert resonance can be utilized to recover the mass density and bulk modulus by employing bubbles as contrast agents. Corresponding advancements in the study of inverse problems via Minnaert resonance in the time domain can be found in \cite{SS23, SSW24}. Recently, the Minnaert resonance has been rigorously established using resolvent operator theory for both time harmonic and time domain cases in \cite{MPS, LS24, LS241}.

	However, the aforementioned studies primarily focus on characterizing the Minnaert resonance with respect to physical parameters or utilizing it to investigate inverse problems. From physical experiments demonstrated in \cite{LAI,LIUCHAN,LIUZHANG}, it is observed that in the sub-wavelength regime, when a high contrast exists between the inclusion and the background medium, the generated wave field exhibits boundary localization and high oscillations near the inclusion boundary. This behavior presents unique challenges and opportunities for further theoretical and practical exploration.

	This study rigorously demonstrates that when a hard radial elastic inclusion is embedded within a soft elastic material in the sub-wavelength regime, by appropriately choosing the incident wave, both the generated internal elastic total wave and the external scattered wave associated with the inclusion and the incident wave are boundary localized. Furthermore, the gradients of the internal elastic wave and the external scattered wave near the boundary of the inclusion exhibit high oscillations. We refer to this highly oscillatory behavior of the wave gradients as {\it surface resonance.}  The boundary localization and surface resonance of the wave scattering phenomena associated with the hard elastic inclusion and the incident wave are termed {\em quasi-Minnaert resonance.}  We present a comprehensive analysis to uncover the physical origin and the mechanism of this new physical phenomenon. It is shown that the delicate high-contrast material structures and the properly tailored incident waves which are coupled together in a subtle manner play a crucial role in ensuring such phenomena.  The quasi-Minnaert resonance is independent of the operating frequency of the incident wave and occurs under the condition that a hard elastic inclusion is situated in a soft elastic background with a suitable incident wave. Consequently, the quasi-Minnaert resonance frequency forms a type of continuous spectrum, in contrast to the discrete Minnaert resonance frequency described in the literature \cite{AFGL,lizou}. Compared to existing mathematical studies of Minnaert resonance, which focus on the mathematical characterization of Minnaert resonance with respect to the physical parameter ratio between the inclusion and the background, our findings initiate an innovative mathematical analysis for characterizing  quasi-Minnaert resonance of wave scattering for hard inclusions in the sub-wavelength regime.

	In the present work, we unveil the underlying physical mechanism of the quasi-Minnaert resonance, which is intrinsically linked to the incident wave configuration and the high contrast parameter of the Lam\'e constants between the inclusion and the background medium. This high contrast parameter characterizes the ratio between the compressional and shear modulus of the hard elastic inclusion and the soft elastic background medium. In Theorem \ref{thm:internal surface localization for scattering problem}, we demonstrate that, for a fixed high contrast parameter of the Lam\'e constants, a careful selection of the incident wave enables boundary localization of the generated internal and external scattered wave fields associated with the incident wave and the fixed material parameter of the hard inclusion. In Corollary \ref{cor:3.2}, we further show that, for a prescribed boundary localization level, appropriate tuning of the high contrast parameter and the incident field can rigorously induce boundary localization phenomena. Additionally, the intricate relationship between the surface resonance of the generated wave field and the incident wave field is rigorously established in Theorem \ref{thm:nabla u in thm}. Specifically, we prove that surface resonance can occur independently of wave field boundary localization by suitably choosing an incident wave that depends only on the high contrast parameter; see Remark \ref{rem:4.1} for details. Finally, based on our analysis of the quasi-Minnaert resonance, we rigorously characterize the stress concentration phenomena of the interior total field and the exterior scattering field in a neighborhood of the boundary of the hard inclusion, as established in Theorem \ref{thm:Eu definition in thm}. These results are derived from a delicate analysis of both the incident wave properties and the material parameter of the system. It is noteworthy that the Minnaert resonance discovered in \cite{lizou} is primarily related to the high contrast in density between the hard inclusion and the elastic background material. In contrast, the quasi-Minnaert resonance introduced in this work relies on the incident wave and the high contrast between the bulk modulus and shear modulus.

    Our arguments for proving boundary localization and surface resonance rely on the layer potential theory and the spectral theory of the Neumann-Poincar\'e operator. To rigorously establish these phenomena, we employ intricate and subtle asymptotic analysis, which is broadly applicable to other types of scattering problems. Furthermore, when quasi-Minnaert resonance occurs, we demonstrate that both the internal total wave and the external scattered wave exhibit stress concentration properties near the boundary of the inclusion. These findings have potential practical applications, particularly in medical treatments, where controlled stress concentration can be leveraged for therapeutic purposes.

	Finally, the remainder of this paper is organized as follows. Section \ref{sec:mathematical formula} presents the mathematical framework and foundational layer potential theory. Section \ref{sec:concentration of u and u^s} rigorously analyzes the boundary localized behavior of the internal total field and external scattered field near the neighborhood of  the boundary of the inclusion. Section \ref{sec:summary of major findings} investigates surface resonance phenomena associated with the internal total field and external scattered field. Furthermore, we rigorously demonstrate the stress concentration phenomena in both the internal total field and the external scattered field.
	
	\section{Mathematical Setup}\label{sec:mathematical formula}
	In this section, we present the mathematical formulation for our subsequent study. Let $D$ be an elastic inclusion characterized by the material parameter $(\tilde{\lambda}, \tilde{\mu},\tilde \rho )$, where $D$ is a bounded Lipschitz domain in $\mathbb{R}^3$ with a connected  complement $\mathbb R^3 \setminus \overline D$.  The real constant Lam\'e parameters $\tilde{\lambda}$ and $\tilde{\mu}$, along with the constant density $\tilde \rho \in \mathbb{R}_{+}$ describe the elastic inclusion $D$. We assume that $\mathbb{R}^3 \setminus \overline {D}$ is filled with a homogeneous elastic medium described by the material parameter $({\lambda}, {\mu}, \rho )$. Both sets of Lam\'e constants satisfy the following strong convexity conditions \cite{Kupradze}: 	
	\begin{equation}\notag
		\tilde \mu>0, \  3 \tilde\lambda+2\tilde \mu>0, \quad  \mbox{and}\quad 	\mu>0, \  3 \lambda+2 \mu>0 .
	\end{equation}
	Here $\rho \in \mathbb{R}_{+}$ represents the density of the homogeneous background $\mathbb{R}^3 \setminus \overline {D}$. Let $\mathbf{u}^i$ be a  time-harmonic incident elastic wave, which is an entire solution to 
	\begin{equation}\label{eq:incident elastic wave}
		\mathcal{L}_{\lambda, \mu} \mathbf{u}^i(\mathbf{x})+\omega^2 \rho \mathbf{u}^i(\mathbf{x})=0 \quad \mbox{in} \quad \mathbb R^3,
	\end{equation}
	where $\omega>0$ is the angular frequency, and the Lam\'e operator $\mathcal{L}_{\lambda, \mu}$ associated with the Lam\'e  parameters $(\lambda, \mu)$ is defined by
	\begin{equation}
		\mathcal{L}_{\lambda, \mu} :=\mu \triangle +(\lambda+\mu) \nabla \nabla \cdot \notag .
	\end{equation}
	The interaction between $\mathbf{u}^i$ and the elastic inclusion $D$ generates the scattered  elastic field $\mathbf{u}^s$ resulting in the total field
	$$
	\mathbf{u}=\mathbf{u}^i+\mathbf{u}^s.
	$$
	The aforementioned elastic scattering problem can be described  by the following system \cite{lizou}:
	\begin{equation}\label{eq:system}
		\begin{cases}\mathcal{L}_{\tilde{\lambda}, \tilde{\mu}} \mathbf{u}(\mathbf{x})+\omega^2 \tilde{\rho} \mathbf{u}(\mathbf{x})=\mathbf{0}, & \mathbf{x} \in D, \\ \mathcal{L}_{\lambda, \mu} \mathbf{u}(\mathbf{x})+\omega^2 \rho \mathbf{u}(\mathbf{x})=\mathbf{0}, & \mathbf{x} \in \mathbb{R}^3 \setminus \overline{D}, \\ \left.\mathbf{u}(\mathbf{x})\right|_{-}=\left.\mathbf{u}(\mathbf{x})\right|_{+}, & \mathbf{x} \in \partial D, \\ \left.\partial_{{\nu},\tilde \lambda, \tilde\mu } \mathbf{u}(\mathbf{x})\right|_{-}=\left.\partial_{\nu,\lambda,\mu }\mathbf{u}(\mathbf{x})\right|_{+}, \quad & \mathbf{x} \in \partial D, \\   \mbox{$\mathbf{u}^s$ satisfies the radiation condition, } & 
		\end{cases}
	\end{equation}
	where the subscripts $\pm$ denote the limits from outside and inside of $D$, respectively. The co-normal derivative $\partial_\nu$  associated with the parameters $(\lambda, \mu)$ appearing in the fourth equation of \eqref{eq:system}  is defined by
	\begin{equation}
		\partial_\nu \mathbf{u}=\lambda(\nabla \cdot \mathbf{u}) \boldsymbol{\nu}+2 \mu\left(\nabla^s \mathbf{u}\right) \boldsymbol{\nu},\notag
	\end{equation}
	with $\nu$ being the outward unit normal to $\partial D$ and  the symmetric gradient operator $\nabla^s$ given by
	\begin{equation}
		\nabla^s \mathbf{u}:=\frac{1}{2}\left(\nabla \mathbf{u}+\nabla \mathbf{u}^{\top}\right).\notag
	\end{equation}
	Here $\nabla \mathbf{u}=\left(\partial_j u_i\right)_{i, j=1}^3$ denotes the gradient of $\mathbf u$ at $\mathbf x$ and the superscript $\top$ denotes the matrix transpose. Using Helmholtz decomposition, the scattered wave field \(\mathbf{u}^s\) can be expressed as a superposition of  compressional and shear waves
	\[
	\mathbf{u}^s = \mathbf{u}_p^s + \mathbf{u}_s^s,
	\]
	where the compressional wave component \(\mathbf{u}_p^s\) is given by
	\[
	\mathbf{u}_{p}^{s} = -\frac{1}{k_{p}^{2}} \nabla (\nabla \cdot \mathbf{u}^{s}),
	\]
	and the shear wave component \(\mathbf{u}_s^s\) is defined as
	\[
	\mathbf{u}_{s}^{s} = \frac{1}{k_{s}^{2}} \nabla \times \nabla \times \mathbf{u}^{s}.
	\]
	In this formulation, \(k_p\) and \(k_s\) are the wave numbers for the compressional and shear waves, respectively, defined as
	\begin{equation}\label{eq:ks kp defination}
		k_p = \frac{\omega}{c_p}, \quad k_s = \frac{\omega}{c_s},
	\end{equation}
	with 
	\begin{equation}\label{eq:cs cp defination}
		c_s=\sqrt{\frac{\mu}{\rho} }, \quad \mbox{and} \quad c_p=\sqrt{\frac{\lambda+2 \mu}{\rho} },
	\end{equation} 
	representing the compressional and shear wave speeds in the medium.
	The radiation condition in \eqref{eq:system} is characterized by the following equations \cite{Kupradze}
	$$
	\begin{array}{r}
		\left(\nabla \times \nabla \times \mathbf{u}^s\right)(\mathbf{x}) \times \frac{\mathbf{x}}{|\mathbf{x}|}- \mathrm{i} k_s \nabla \times \mathbf{u}^s(\mathbf{x})=\mathcal{O}\left(|\mathbf{x}|^{-2}\right), \\
		\frac{\mathbf{x}}{|\mathbf{x}|} \cdot\left[\nabla\left(\nabla \cdot \mathbf{u}^s\right)\right](\mathbf{x})-\mathrm{i} k_p \nabla \mathbf{u}^s(\mathbf{x})=\mathcal{O}\left(|\mathbf{x}|^{-2}\right),
	\end{array}
	$$
	as $|\mathbf{x}| \rightarrow+\infty$, where $\rmi$ is the imaginary unit.	It is pointed out that the corresponding physical parameters \(\tilde{k}_s, \tilde{k}_p, \tilde{c}_s, \tilde{c}_p\) associated with the elastic inclusion \(D\) can be defined in a similar way as in \eqref{eq:ks kp defination} by substituting \((\lambda, \mu, \rho)\) with \((\tilde{\lambda}, \tilde{\mu}, \tilde{\rho})\).

	This paper focuses primarily on a configuration in which the hard inclusion $D$ is embedded within a soft elastic material, exemplified by scenarios such as lead inclusions coated with silicone rubber\cite{LIUZHANG}. In this configuration, the corresponding physical parameters in different inclusions follow the subsequent relationship:
	\begin{equation}\label{eq:lambda mu rho}
		\tilde{\lambda}=\frac{1}{\delta} \lambda, \quad \tilde{\mu}=\frac{1}{\delta} \mu, \quad \tilde{\rho}=\frac{1}{\epsilon} \rho,
	\end{equation}
	where $\delta \ll 1$ and $\epsilon \ll 1$ are given positive constants. In fact, $\delta$ describes the contrast between Lam\'e constants of two different media, whereas $\epsilon$ characterizes the contrast of two different densities. 
	Subsequently, we introduce another parameter $\tau$ that describes the contrast in terms of wave velocities among different media, namely
	\begin{equation}\label{eq:tau defination}
		\tau=\frac{c_s}{\tilde{c}_s}=\frac{c_p}{\tilde{c}_p}=\sqrt{\delta / \epsilon}.
	\end{equation}
	Furthermore, we assume that the contrast $\tau$ satisfies
	\begin{equation}\label{eq:tau satisfies}
		\tau=\sqrt{\delta / \epsilon} < 1.
	\end{equation}
	The assumption in \eqref{eq:tau satisfies} aligns with the physical setup, where the wave speed within the hard inclusion $D$ is higher than that in the surrounding soft material. In this study, we explore the sub-wavelength scenario, i.e.,
	\begin{equation}\label{eq:ass sub-wave}
		\omega \cdot {\rm diam} (\Omega) \ll 1.
	\end{equation}
	Indeed, the assumption \eqref{eq:ass sub-wave} indicates that the diameter of the inclusion $D$ is smaller than the operating wavelength of the incident wave. Throughout the paper, we assume that the material parameter $(\lambda,\mu,\rho)$ describing the homogeneous elastic medium satisfies
	\begin{equation}\label{eq:lambda mu rho O1}
		\lambda=\mathcal{O}(1),\quad \mu=\mathcal{O}(1),\quad \rho=\mathcal{O}(1).
	\end{equation}
	By applying a suitable coordinate  transformation, we assume that the diameter of the inclusion $D$ is of order 1. Hence, in view of \eqref{eq:ass sub-wave}, we know that	
	\begin{equation}\label{eq:assp omega}
		\omega=o(1), 
	\end{equation}
	which implies that the compressional and shear wave numbers satisfy 
	\begin{align}\label{eq:kp small}
		k_s=o(1)\quad \mbox{and}\quad k_p=o(1). 
	\end{align}
	Furthermore, according to the relationship in \eqref{eq:tau satisfies}, it follows that 
	\begin{align}\label{eq:kp in small}
		\tilde{k}_s=o(1)\quad \mbox{and}\quad \tilde{k}_p=o(1),
	\end{align}
	with
	\begin{equation}\label{eq:tilde ks relationship with ks}
		\tilde{k}_s=\tau k_s, \quad \tilde{k}_p=\tau k_p.
	\end{equation}

	This paper primarily employs potential theory to analyze the system \eqref{eq:system}. We begin by introducing the potential theory relevant to the Lam\'e system. The fundamental solution $\mathbf{\Gamma}^\omega = \left(\Gamma_{i, j}^\omega\right)_{i, j=1}^3$ for the operator $\mathcal{L}_{\lambda, \mu} + \omega^2\rho$ in three dimensions is given in \cite{ABG} as
	\begin{equation}
		\left(\Gamma_{i, j}^\omega\right)_{i, j=1}^3(\mathbf{x}) = -\frac{\delta_{ij}}{4 \pi \mu |\mathbf{x}|} e^{\mathrm{i} k_s |\mathbf{x}|} + \frac{1}{4 \pi \omega^2 \rho} \partial_i \partial_j \frac{e^{\mathrm{i} k_p |\mathbf{x}|} - e^{\mathrm{i} k_s |\mathbf{x}|}}{|\mathbf{x}|},\notag
	\end{equation}
	where $\delta_{ij}$ is the Kronecker delta function. Here, $k_\alpha$, for $\alpha = p, s$, is defined by \eqref{eq:ks kp defination}. It can be shown that $\Gamma_{i j}^\omega$ has the following series representation (cf. \cite{ABG}):
	\begin{align}
		\Gamma_{i j}^\omega(\mathbf{x}) = & -\frac{1}{4 \pi \rho} \sum_{n=0}^{+\infty} \frac{\mathrm{i}^n}{(n+2) n!}\left(\frac{n+1}{c_s^{n+2}} + \frac{1}{c_p^{n+2}}\right) \omega^n \delta_{ij} |\mathbf{x}|^{n-1} \notag\\
		& + \frac{1}{4 \pi \rho} \sum_{n=0}^{+\infty} \frac{\mathrm{i}^n (n-1)}{(n+2) n!} \left(\frac{1}{c_s^{n+2}} - \frac{1}{c_p^{n+2}}\right) \omega^n |\mathbf{x}|^{n-3} x_i x_j.\notag
	\end{align}
	In particular, when $\omega = 0$, we denote $\mathbf{\Gamma}^0$ as $\mathbf{\Gamma}$ for simplicity, and its expression (cf. \cite{ABG}) is given by
	\begin{equation}\notag
		\Gamma_{i, j}^0(\mathbf{x}) = -\frac{1}{8 \pi} \left(\frac{1}{\mu} + \frac{1}{\lambda + 2\mu}\right) \frac{\delta_{ij}}{|\mathbf{x}|} - \frac{1}{8 \pi} \left(\frac{1}{\mu} - \frac{1}{\lambda + 2\mu}\right) \frac{x_i x_j}{|\mathbf{x}|^3}.
	\end{equation}
	The single-layer potential associated with the fundamental solution $\mathbf{\Gamma}^\omega$ is defined as
	\begin{equation}\label{eq:single layer potential}
		\mathcal{S}_{\partial D}^\omega[\boldsymbol{\varphi}](\mathbf{x}) = \int_{\partial D} \mathbf{\Gamma}^\omega(\mathbf{x} - \mathbf{y}) \boldsymbol{\varphi}(\mathbf{y}) \, \mathrm{d}s(\mathbf{y}), \quad \mathbf{x} \in \mathbb{R}^3,
	\end{equation}
	for $\boldsymbol{\varphi} \in L^2(\partial D)^3$. On $\partial D$, the co-normal derivative of the single-layer potential satisfies the following jump relation:
	\begin{equation}\label{eq:jump relation}
		\left.\partial_\nu \mathcal{S}_{\partial D}^\omega[\boldsymbol{\varphi}]\right|_{\pm}(\mathbf{x}) = \left( \pm \frac{1}{2} \mathcal{I} + \mathcal{K}_{\partial D}^{\omega, *} \right)[\boldsymbol{\varphi}](\mathbf{x}), \quad \mathbf{x} \in \partial D,
	\end{equation}
	where
	\begin{equation}\label{eq:np operator}
		\mathcal{K}_{\partial D}^{\omega, *}[\boldsymbol{\varphi}](\mathbf{x}) = \text{p.v.} \int_{\partial D} \partial_{\nu_{\mathbf{x}}} \mathbf{\Gamma}^\omega(\mathbf{x} - \mathbf{y}) \boldsymbol{\varphi}(\mathbf{y}) \, \mathrm{d}s(\mathbf{y}).
	\end{equation}
	Here, p.v. refers to the Cauchy principal value. Notably, the operator $\mathcal{K}_{\partial D}^{\omega, *}$, defined in equation \eqref{eq:np operator}, is known as the Neumann-Poincar\'e (N-P) operator. This operator plays a significant role in the study of metamaterials within elasticity theory.

	With the aid of the potential theory described earlier, the solution to the system \eqref{eq:system} can be expressed as
	\begin{equation}\label{eq:total field}
		\mathbf{u} = 
		\begin{cases}
			\tilde{\mathcal{S}}_{\partial D}^\omega[	\boldsymbol{\varphi_1}](\mathbf{x}), & \mathbf{x} \in D, \\
			\mathcal{S}_{\partial D}^\omega[	\boldsymbol{\varphi_2}](\mathbf{x}) + \mathbf{u}^i, & \mathbf{x} \in \mathbb{R}^3 \setminus \overline{D},
		\end{cases}
	\end{equation}
	where the density functions $	\boldsymbol{\varphi_1},\ 	\boldsymbol{\varphi_2} \in L^2(\partial D)^3$ are determined by the transmission conditions across $\partial D$ in \eqref{eq:system}. Here, the operator $\tilde{\mathcal{S}}_{\partial D}^\omega$ denotes the single-layer potential given by \eqref{eq:single layer potential} associated with the parameters $(\tilde{\lambda}, \tilde{\mu})$ and $\tilde{\rho}$. Based on the transmission conditions in \eqref{eq:system} and the jump relation \eqref{eq:jump relation}, it is straightforward to derive the equivalent integral equations for \eqref{eq:system}, given by
	\begin{equation}\label{eq:scattering problem equation}
		\mathcal{A}(\omega, \delta)[\bmf{\Phi}](\mathbf{x}) = \bmf{F}(\mathbf{x}), \quad \mathbf{x} \in \partial D,
	\end{equation}
	where
	\begin{equation}
		\mathcal{A}(\omega, \delta) = \begin{pmatrix}
			\tilde{\mathcal{S}}_{\partial D}^\omega & -\mathcal{S}_{\partial D}^\omega \notag\\
			-\frac{I}{2} + \tilde{\mathcal{K}}_{\partial D}^{\omega, *} & -\frac{I}{2} - \mathcal{K}_{\partial D}^{\omega, *}
		\end{pmatrix}, \quad 
		\bmf{\Phi} = \begin{pmatrix}
			\boldsymbol{\varphi_1} \\
			\boldsymbol{\varphi_2}
		\end{pmatrix}, \quad 
		\bmf{F} = \begin{pmatrix}
			\mathbf{u}^i \\
			\partial_{\boldsymbol{\nu}} \mathbf{u}^i
		\end{pmatrix}.
	\end{equation}
	Here the N-P operator $\tilde{\mathcal{K}}_{\partial D}^{\omega, *}$ is defined by \eqref{eq:np operator} associated with the parameters $(\tilde{\lambda}, \tilde{\mu})$ and $\tilde{\rho}$.   For the subsequent analysis, we define the spaces $\mathcal{H} = L^2(\partial D)^3 \times L^2(\partial D)^3$ and $\mathcal{H}^1 = H^1(\partial D)^3 \times L^2(\partial D)^3$. The operator $\mathcal{A}(\omega, \delta)$ maps from $\mathcal{H}$ to $\mathcal{H}^1$. 
	
	In this paper we shall investigate the boundary localization and surface resonance for the elastic scattering \eqref{eq:system}. In the following we introduce the corresponding definitions.

	\begin{defn}\label{def:surface localized}
     Let $D$ denote a bounded Lipschitz domain with a connected complement, and consider the internal total field $\mathbf{u}|_{D}$ along with the external scattered field $\mathbf{u}^s|_{\mathbb{R}^3 \setminus \overline{D}}$ associated with the scattering problem \eqref{eq:system} and an incident wave $\mathbf{u}^i$. Let $B_R$ denote a ball of radius $R$ centered at the origin in $\mathbb{R}^3$, satisfying $D \subseteq B_R$. For a fixed, sufficiently small $\xi_1, \xi_2 \in \mathbb{R}_{+}$, we define the interior and exterior neighborhoods relative to the boundary of $D$ as follows:  
		\begin{align}\label{eq:Mdef}
			\mathcal{M}_{-}^{\xi_1}(\partial D) := \{ \mathbf{x} \in D \,\big|\, \mathrm{d}_{\partial D}(\mathbf{x}) < \xi_1 \},
			\quad 
			\mathcal{M}_{+}^{\xi_2}(\partial D) := \{ \mathbf{x} \in  B_R \setminus \overline D \,\big|\, \mathrm{d}_{\partial D}(\mathbf{x}) < \xi_2 \},
		\end{align}   
		where $\mathrm{d}_{\partial D}(\mathbf{x}) := \inf\limits_{\mathbf{y} \in \partial D} \|\mathbf{x} - \mathbf{y}\|$ denotes the Euclidean distance to $\partial D$. The internal total field $\mathbf{u}|_{D}$ is termed  {\it interior boundary localized}, while the external scattered field  $\mathbf{u}^s|_{\mathbb{R}^3 \setminus \overline{{D}}}$ is referred to as  {\it exterior boundary localized}, provided there exist sufficiently small parameters $\xi_1, \xi_2, \varepsilon \in \mathbb{R}_{+}$ corresponding to each case such that
		\begin{align}\label{eq:220}
			\frac{\|\mathbf{u}\|_{L^{2}(D \setminus \mathcal{M}_{-}^{\xi_1}(\partial D))^{3}}}{\|\mathbf{u}\|_{L^{2}(D)^{3}}} \leqslant \varepsilon,
			\quad
			\frac{\|\mathbf{u}^s\|_{L^{2}((B_R \setminus \overline D) \setminus \mathcal{M}_{+}^{\xi_2}(\partial D))^{3}}}{\|\mathbf{u}^s\|_{L^{2}(B_R \setminus \overline D)^{3}}} \leqslant \varepsilon.
		\end{align}
		Here the quantity $\varepsilon $ characterizes the level of the boundary localization.  
	\end{defn}

	\begin{defn}\label{def:surface resonant}
		Consider the internal total field $\mathbf{u}|_{D}$ and the external scattered field $\mathbf{u}^s|_{\mathbb{R}^3 \setminus \overline{{D}}}$ corresponding to the scattering problem \eqref{eq:system} associated with  incident wave $\mathbf{u}^i$, where $D$ represents an inclusion with a bounded Lipschitz boundary. We recall that $\mathcal{M}_{-}^{\xi_1}(\partial D) $ and $\mathcal{M}_{+}^{\xi_2}(\partial D)$ are defined in \eqref{eq:Mdef}. If 
		\begin{align}\notag
			\frac{\|\nabla \mathbf{u} \|_{L^2(\mathcal{M}_{-}^{\xi_1}(\partial D) )^3}}{\|\mathbf{u}^i\|_{L^2(D)^3} } \gg 1, 
			\qquad \mbox{and} \qquad
			\frac{\|\nabla \mathbf{u}^s \|_{L^2(\mathcal{M}_{+}^{\xi_2}(\partial D))^3} }{\|\mathbf{u}^i\|_{L^2(D)^3} } \gg 1,
		\end{align}
		hold, then we say  that  $\mathbf{u}|_{D}$ and $\mathbf{u}^s|_{\mathbb{R}^3 \setminus \overline{{D}}}$ are surface resonant.
		
	\end{defn}

\begin{rem}

The left-hand sides of the two inequalities in \eqref{eq:220} represent the ratios between the $L^2$ norm of the wave field slightly away from $\partial D$ and the corresponding $L^2$ norm of the wave field in $D$. These ratios are shown to be small in terms of the boundary localization level parameter $\varepsilon$, introduced in Definition \ref{def:surface localized}, which is assumed to be sufficiently small. Notably, the denominators and numerators of these terms can be divided by a common factor---the $L^2$ norm of the incident wave $\mathbf{u}^i$ in $D$---without altering the inequalities. This normalization step is justified by the linearity of the scattering problem \eqref{eq:system}, which involves linear partial differential equations (PDEs). In our subsequent analysis, when both boundary localization and surface resonance occur for the internal elastic wave $\mathbf{u}|_D$ and the scattered wave field $\mathbf{u}^s|_{B_R\backslash D}$ corresponding to \eqref{eq:system} with the incident wave $\mathbf{u}^i$ and the hard inclusion $D$, we refer to $D$ as a \textit{quasi-Minnaert resonator} in Definition \ref{def:quasi minnaert}. Consequently, $D$ serves as an operating inclusion in our study, particularly in the sub-wavelength regime where the size of $D$ is much smaller than the impinging wavelength. Since $D$ is small, the $L^2$ norm of $\mathbf{u}^i$ within $D$ is also small. Furthermore, both boundary localization and surface resonance are localized around $D$, which further justifies the necessity of normalizing the $L^2$ norm of $\mathbf{u}^i$ within $D$.

\end{rem}

	Building on Definitions~\ref{def:surface localized} and~\ref{def:surface resonant}, we propose the definition of the quasi-Minnaert resonance.

	\begin{defn}\label{def:quasi minnaert}
    
   	Consider the elastic scattering problem described by system \eqref{eq:system} for a bounded Lipschitz domain \( D \), impinged by an incident wave \( \mathbf{u}^i \) at the operating frequency \( \omega \). If the generated internal total field \( \mathbf{u}|_{D} \) and the external scattered field \( \mathbf{u}^s|_{\mathbb{R}^3 \setminus \overline{D}} \) of \eqref{eq:system} satisfy both boundary localization and surface resonance in the sense of Definitions~\ref{def:surface localized} and~\ref{def:surface resonant}, respectively, then the frequency \( \omega \) is referred to as the \textit{quasi-Minnaert resonance frequency} associated with the incident wave \( \mathbf{u}^i \). In this context, the inclusion \( D \) is termed the \textit{quasi-Minnaert resonator} corresponding to the incident wave $ \mathbf{u}^i $.
   \end{defn}

	\begin{rem}
		We emphasize that the quasi-Minnaert resonance introduced in Definition~\ref{def:quasi minnaert} depends on the incident wave \( \mathbf{u}^i \) and the material parameter $D$. Consequently, this type of resonance is modulated by the incident wave \( \mathbf{u}^i \) and physical parameters characterizing the hard elastic inclusion $D$, a relationship that will be further clarified in Theorems~\ref{thm:internal surface localization for scattering problem} and~\ref{thm:nabla u in thm}.  In the sub-wavelength regime, under the physical setup of a hard elastic inclusion embedded in a soft elastic background, we rigorously prove that by choosing an appropriate incident wave \( \mathbf{u}^i \) as described by \eqref{eq:incident wave in T_n^m}, both the corresponding internal wave field and the exterior scattered wave are boundary localized. Furthermore, surface resonance is established in Theorem~\ref{thm:nabla u in thm} when the incident wave \( \mathbf{u}^i \) takes the form given by \eqref{eq:u^i 4.1}.  In contrast to existing literature \cite{AFGL,lizou} on discrete Minnaert resonance frequencies, which are characterized by the high contrast ratio of the material parameter and are independent of the incident wave \( \mathbf{u}^i \), the newly proposed quasi-Minnaert resonance frequencies form a continuous set in the sub-wavelength regime. The quasi-Minnaert resonance is closely related to explaining physical phenomena \cite{LAI,LIUCHAN,LIUZHANG} associated with boundary localization and high surface oscillations of the generated wave field.

	\end{rem}

	\section{Boundary localization  of the interior total field and exterior scattered wave field}\label{sec:concentration of u and u^s}

	In this section, we introduce some spectral properties of the single-layer potential and the N-P operator corresponding to a unit ball in $\mathbb{R}^3$. Utilizing these spectral properties, we select a suitable incident wave $\mathbf{u}^i$ given by \eqref{eq:incident wave in T_n^m}, which satisfies \eqref{eq:incident elastic wave}. This ensures that the corresponding internal total field $\mathbf{u}$ and external scattered field $\mathbf{u}^s$ of \eqref{eq:system}, associated with $\mathbf{u}^i$, exhibit boundary localization within the interior and exterior of the inclusion $D$, respectively, as stated in Theorem \ref{thm:internal surface localization for scattering problem}.

	In the following discussion, we introduce relevant notation and formulas. Let $\mathbb{N}_+$ denote the set of positive integers, and define $\mathbb{N}_0 = \mathbb{N}_+ \cup \{0\}$. The spherical harmonic functions $Y_n^m\left(\theta, \varphi\right)$ (cf.\cite{CK}) are given by
	\begin{align}\label{eq:ynm and cnm def}
		Y_n^m\left(\theta, \varphi\right) := C_n^m P_n^{|m|}(\cos\theta) e^{\rmi m \varphi},\quad  C_n^m = \sqrt{\frac{2n+1}{4\pi} \frac{(n-|m|)!}{(n+|m|)!}} ,
	\end{align}
	where $n \in \mathbb{N}_0$ and $-n \leqslant m \leqslant n$. Denote $j_n(z)$ and $h_n(z)$ as the spherical Bessel and Hankel functions of the first kind of order $n$, respectively.  For any fixed $n \in \mathbb{N}_0$, if $0 < |z| \ll 1$, the following expansions hold (cf. \cite[(2.32)]{CK}):
	\begin{subequations}
		\begin{align}
			j_n(z) &= \frac{z^n}{(2n+1)!!} \left(1 - \frac{z^2}{2(2n+3)} + \frac{z^4}{8(2n+3)(2n+5)} + \mathcal{O}(z^5)\right), \label{eq:jn expansion}\\
			h_0(z) &= \frac{1}{\mathrm{i} z}\left(1 + \mathrm{i} z - \frac{z^2}{2} - \frac{\mathrm{i} z^3}{6} + \frac{z^4}{24} + \mathcal{O}(z^5)\right), \notag\\
			h_1(z) &= \frac{1}{\mathrm{i} z^2}\left(1 + \frac{z^2}{2} + \frac{\mathrm{i} z^3}{3} - \frac{z^4}{8} + \mathcal{O}(z^5)\right), \notag\\
			h_n(z) &= \frac{(2n-1)!!}{\mathrm{i} z^{n+1}} \left(1 - \frac{z^2}{2(2n-1)} + \frac{z^4}{8(2n-1)(2n-3)} + \mathcal{O}(z^5)\right), \quad n \geqslant 2. \label{eq:hn expansion}
		\end{align}
	\end{subequations}
	From the series representations of $h_n(z)$, the following recurrence relations hold (cf. \cite{NIST}):
	\begin{align}\label{eq:hn qiudao}
		h_n^{\prime}(z) &= h_{n-1}(z) - \frac{(n+1)}{z} h_n(z), \quad n=1, 2, 3, \dots, \notag \\
		h_n^{\prime}(z) &= -h_{n+1}(z) + \frac{n}{z} h_n(z), \quad n=0, 1, 2, \dots,\\
		h_{n+1}(z)&=\frac{2n+1}{z}h_n(z)-h_{n-1}(z), \quad n=1, 2, 3, \dots .\notag
	\end{align}
	These differentiation formulas are also applicable to $j_n(z)$.

	Lemma \ref{lem:spectral NP} states some spectral properties of the single-layer potential and the N-P operator provided that the elastic inclusion $D$ is a ball associated with the material parameter $(\tilde{\lambda}, \tilde{\mu},\tilde \rho )$ and $\mathbb R^3 \setminus \overline D$ is a homogeneous elastic medium associated with the material parameter $({\lambda}, {\mu}, \rho )$, where 	Lemma \ref{lem:spectral NP}  can be proved by modifying the proofs of \cite[Theorem 3.1-3.2]{DLL2020}. Hence the detailed proof is omitted. 
	\begin{lem}\label{lem:spectral NP}
		The set $\left(\mathcal{I}_n^m, \mathcal{T}_n^m, \mathcal{N}_n^m\right)$, denoting the vectorial spherical harmonics of order $n$, constitutes an orthogonal basis of $L^2(\mathbb{S}^2)^3$, where 
		\begin{align}\label{eq:Tnm definition}
			\mathcal{I}_n^m (\theta, \varphi) &=\nabla_{\mathbb{S}^2} Y_{n+1}^m(\theta, \varphi)+(n+1) Y_{n+1}^m(\theta, \varphi) \nu, \quad n \geqslant 0, n+1 \geqslant m \geqslant-(n+1), \notag\\
			\mathcal{T}_n^m (\theta, \varphi) &=\nabla_{\mathbb{S}^2} Y_n^m(\theta, \varphi) \wedge \nu, \quad n \geqslant 1, n \geqslant m \geqslant-n, \\
			\mathcal{N}_n^m (\theta, \varphi) &=-\nabla_{\mathbb{S}^2} Y_{n-1}^m(\theta, \varphi)+n Y_{n-1}^m(\theta, \varphi) \nu, \quad n \geqslant 1, n+1 \geqslant m \geqslant-(n+1).\notag 
		\end{align}
		Denote $D_R$ by the  ball centered at the origin with radius $R \in \mathbb{R}_{+}$ in $\mathbb{R}^3$. 
		We have the following spectral property of the single-layer potential operator ${\mathcal{S}}_{\partial D_R}^{\omega}$ defined in \eqref{eq:single layer potential}:
		\begin{equation}
			\mathcal{S}_{\partial D_R}^{\omega}\left[\mathcal{T}_n^m\right](\mathbf{x})=-\frac{\mathrm{i} {k}_s R^2 j_{n}(k_s R) h_{n}(k_s R)}{\mu}\mathcal{T}_n^m, \quad \mathbf{x} \in \partial D_R.\notag
		\end{equation}
		Moreover, the following two identities hold
		\begin{equation}\label{eq:S partial outside ball}
			{\mathcal{S}}_{\partial D_R}^{\omega}\left[\mathcal{T}_n^m\right](\mathbf{x})=-\frac{\mathrm{i} {k}_s R^2 j_{n}(k_s R) h_n\left({k}_s|\mathbf{x}|\right)}{\mu} \mathcal{T}_n^m, \quad\mathbf{x} \in \mathbb{R}^3 \setminus D_R,
		\end{equation}
		and 
		\begin{equation}\label{eq:S Tnm in ball}
			\tilde{\mathcal{S}}_{\partial D_R}^{\omega}\left[\mathcal{T}_n^m\right](\mathbf{x})=-\frac{\mathrm{i} \tilde{k}_s R^2 h_{n} (\tilde{k}_{ s} R)j_n\left(\tilde{k}_s|\mathbf{x}|\right)}{\tilde\mu} \mathcal{T}_n^m, \quad \mathbf{x} \in D_R,
		\end{equation}	
		where ${k_s}$, ${k_p}$ defined in \eqref{eq:ks kp defination} and $\tilde{k}_s$, $\tilde{k}_p$ defined in \eqref{eq:tilde ks relationship with ks}. The spectral system of the N-P operator ${\mathcal{K}}_{\partial D}^{\omega,*}$ is given as follows
		\begin{equation}\label{eq:K* in Tnm}
			\left({\mathcal{K}}_{\partial D}^{\omega,*}\right)\left[\mathcal{T}_n^m\right]=\lambda_{1, n} \mathcal{T}_n^m, 
		\end{equation}
		where
		\begin{equation}\notag 
			\lambda_{1, n}=-\rmi  {k}_s R j_{n}(k_s R)\left({k}_s R h^{\prime}_{n}(k_s R)-h_{n}(k_s R)\right)-\frac{1}{2}.
		\end{equation}
	\end{lem}

Next, we present the appropriate choice of incident waves relevant to our study. For any $n\in \mathbb N$, let 
\begin{equation}\label{eq:incident wave in T_n^m}
	\mathbf{u}^{i}(\bmf x)=\sum_{m=-n}^n f_{n, m} j_n\left(k_s|\mathbf{x}|\right) \mathcal{T}_n^m (\theta, \varphi),
\end{equation}
where $k_s$ and $ \mathcal{T}_n^m$ are defined by \eqref{eq:ks kp defination} and  \eqref{eq:Tnm definition}, respectively,  nonzero vector $(f_{n, -n},\ldots, f_{n, n})\in \mathbb C^{2n+1}$.  
According to \cite{BenSinbook,Dassions,DLL2020}, $\mathbf{u}^{i}$ given by \eqref{eq:incident wave in T_n^m} is an entire solution to \eqref{eq:incident elastic wave}, which is the incident wave for investigating the boundary localization of the internal total wave field $\mathbf u|_{D}$ and exterior scattered wave field $\mathbf u^s|_{\mathbb{R}^3\setminus \overline D}$ to \eqref{eq:scattering problem equation} associated with    $\mathbf{u}^{i}$ in Theorem \ref{thm:internal surface localization for scattering problem}. In the following lemma, we shall derive asymptotic  expansions for $\mathbf u|_{D}$ and $\mathbf u^s|_{\mathbb R^3 \setminus \overline D}$ with respect to the sub-wavelength angular frequency $\omega$ and the high contrast parameter $\delta$ defined by \eqref{eq:lambda mu rho}.

\begin{lem}\label{lem:total field and scatter field}
	Consider the elastic scattering problem \eqref{eq:system} associated with the incident wave $	\mathbf{u}^{i}$ given by \eqref{eq:incident wave in T_n^m} and the elastic inclusion $(D;\tilde{\lambda},\tilde{\mu},\tilde{\rho})$ embedded in a homogeneous elastic medium $(\mathbb R^3 \setminus \overline D;\lambda,\mu,\rho)$, where $D$ is a unit ball centered at the origin in $\mathbb R^3$. Under the assumptions \eqref{eq:lambda mu rho}-\eqref{eq:tau satisfies} and \eqref{eq:assp omega},  the internal total field $\mathbf u$ in $D$ and the scattered field $\mathbf u^s$ in $\mathbb R^3\setminus \overline D$   to  \eqref{eq:system} have the following asymptotic  expansions with respect to $\omega$ and $\delta$ as follows 
	
	\begin{align}
		\mathbf{u}|_D
		&=\sum_{m=-n}^{n}\frac{f_{n,m}\delta\rho^{\frac{n}{2}}|\bmf x|^n\omega^n}{(2n-1)!![\delta(n+2)+n-1]\mu^{\frac{n}{2}}}\left(1+\mathcal{O}\left(\frac{\rho\omega^2}{(2n+1)\mu}\right)\right)\mathcal{T}_n^m,\label{eq:u single potential inside the unit sphere} \\
		\mathbf u^s|_{\mathbb R^3 \setminus \overline D} 
		&=\sum_{m=-n}^n \frac{-f_{n,m}(n-1)(1-\delta) \rho^{\frac{n}{2}}\omega^n }{(2 n+1) ! ![\delta(n+2)+n-1]\mu^{\frac{n}{2}}|\bmf x|^{n+1} }\left(1+\mathcal{\mathcal{O}}\left( \frac{[(1-n)+\delta(n+1)]\rho\omega^{2}}{(2n+3)(n-1)(1-\delta)\mu}\right)\right)\mathcal{T}_n^m. \label{eq:u^s expansion} 
	\end{align}

\end{lem}

\begin{proof}

	In view of  \eqref{eq:scattering problem equation} and \eqref{eq:incident wave in T_n^m}, using the orthogonality of the functions $\mathcal{T}_n^m$, $\mathcal{I}_n^m$ and $\mathcal{N}_n^m$, the density functions $\boldsymbol{\varphi_1}\in L^2 (\mathbb S^2 )^3 $ and $\boldsymbol{\varphi_2}\in L^2 (\mathbb S^2 )^3$ in \eqref{eq:scattering problem equation}  can be chosen  as    
	\begin{align}
		\boldsymbol{\varphi_1}&=\sum_{m=-n}^n \varphi_{1, n,m} \mathcal{T}_n^m, \quad 						\boldsymbol{\varphi_2}= \sum_{m=-n}^n \varphi_{2, n,m} \mathcal{T}_n^m, \notag
	\end{align} 
	where $\varphi_{j, n,m} $ $(j=1,2)$ are constants to be determined.

	From the jump relation \eqref{eq:jump relation}, using \eqref{eq:S Tnm in ball}, \eqref{eq:K* in Tnm} and the fact that $\left(\mathcal{I}_n^m, \mathcal{T}_n^m, \mathcal{N}_n^m\right)$ forms a orthogonal basis of $L^2 (\mathbb S^2 )^3$, the integral equation \eqref{eq:scattering problem equation}  can be written in the following algebraic system
	\begin{equation}\label{eq:BX=b}
		\boldsymbol{A}_n\boldsymbol{x}_{n,m}=\boldsymbol{b}_{n,m},
	\end{equation}
	where
	\begin{equation}\label{eq:B X b defination}
		\boldsymbol{A}_n=	\left[\begin{array}{ll}
			a_{11} & a_{12} \\
			a_{21} & a_{22}
		\end{array}\right],\quad \boldsymbol{x}_{n,m}=\left[\begin{array}{l}
			\varphi_{1, n,m} \\
			\varphi_{2, n,m}
		\end{array}\right],\quad\boldsymbol{b}_{n,m}=\left[\begin{array}{c}
			f_{n,m} j_n\left(k_s \right) \\
			g_{n, m}
		\end{array}\right],
	\end{equation}
	with
	\begin{align}
		a_{11}&=-\frac{\mathrm{i} \tilde{k}_s  j_n(\tilde{k}_s ) h_n(\tilde{k}_s )}{\tilde\mu}, \quad a_{12}=\frac{\mathrm{i} k_s  j_n(k_s ) h_n(k_s )}{\mu}, \quad g_{n, m}=f_{n,m} \mu\left(k_s  j_n^{\prime}\left(k_s \right)-j_n\left(k_s \right)\right), \notag\\
		a_{21}&=-\mathrm{i} \tilde{k}_s  h_n(\tilde{k}_s )(\tilde{k}_s  j_n^{\prime}(\tilde{k}_s )-j_n(\tilde{k}_s )), \quad a_{22}=\mathrm{i} {k}_s  j_n\left({k}_s \right)\left({k}_s  h_n^{\prime}\left({k}_s \right)-h_n\left({k}_s \right)\right). \notag
	\end{align}
	Here $k_s$, $k_p$, $\tilde{k}_s$ and $\tilde{k}_p$ are defined in \eqref{eq:ks kp defination} and \eqref{eq:lambda mu rho}, respectively. Due to $\omega \ll 1$ and $\delta \ll 1$, it yields that \eqref{eq:kp small} and \eqref{eq:kp in small}, which imply that the system matrix $\boldsymbol{ A}_n $ in \eqref{eq:BX=b} has the following absolutely convergent series as follows 
	\begin{equation}\label{eq:B expansion}
		\boldsymbol A_n=	\boldsymbol{A_0}+\omega^2\boldsymbol{A_2}+\omega^4\boldsymbol{A_4}+\mathcal{O}\left(\frac{\omega^6}{n^7}\right),
	\end{equation}
	where we utilize \eqref{eq:jn expansion} and \eqref{eq:hn expansion}. Here
	
	\begin{align}
		\label{eq:A0 A2 defination}
		\boldsymbol{A_0}&=\frac{1}{2n+1}\left[\begin{array}{ll}
			-\frac{\delta}{\mu} & \quad\frac{1}{\mu} \\
			1-n & -(n+2)
		\end{array}\right],\quad
		\boldsymbol{A_2}=\frac{\rho}{(2n+1)(2n+3)(-2n+1)\mu}	\left[\begin{array}{ll}
			\frac{2\delta\tau^2}{\mu} & -\frac{ 2}{\mu} \\
			-\tau^2 &\quad 1
		\end{array}\right], \\
		\label{eq:A4 def}
		\boldsymbol{A_4}&=\frac{3\rho^2}{\mu^2(2n-1)(2n+1)(2n+3)(2n+5)(2n-3)}	\left[\begin{array}{ll}
			\frac{2\delta\tau^4}{\mu} & -\frac{2}{\mu} \\
			-\tau^4 &\quad 1
		\end{array}\right].
	\end{align}
	Since $\delta \in \mathbb R_+$, it can be directly to see that
	$$
	\det(\boldsymbol A_0)=\frac{\delta(n+2)+n-1}{\mu(2n+1)^2}\neq 0,
	$$
	for $n\in \mathbb N$. Therefore, after direct calculations, $\boldsymbol A_0^{-1}$ can be expressed as
	\begin{equation}\label{eq:A0 ni}
		\boldsymbol{A_0^{-1}}=\frac{2n+1}{\delta(n+2)+n-1}	\left[\begin{array}{ll}
			-\mu(n+2)& -1 \\
			\mu(n-1)& -\delta
		\end{array}\right].
	\end{equation}
	Substituting \eqref{eq:B expansion} into \eqref{eq:BX=b}, using \eqref{eq:A0 A2 defination}, \eqref{eq:A4 def} and \eqref{eq:A0 ni}, it yields that			
	\begin{align}\label{eq:x_nm}
		\boldsymbol{x}_{n,m}=&\left(\boldsymbol  {I}+\omega^2\left(\boldsymbol{A_0^{-1}}\boldsymbol{A_2}+\omega^2\boldsymbol{A_0^{-1}}\boldsymbol{A_4}+\mathcal{O}\left(\frac{\omega^4}{n^7}\right)\right)\right)^{-1}\boldsymbol{A_0^{-1}}\boldsymbol{b}_{n,m}\\
		=&\left(\boldsymbol{I}-\omega^2\boldsymbol{A_0^{-1}}\boldsymbol{A_2}+\mathcal{O}(\frac{\omega^4}{n^4})\right)\boldsymbol{A_0^{-1}}\boldsymbol{b}_{n,m}\notag\\
		=&\left(\begin{array}{c}
			c_{1,n,m}	 \\
			c_{2,n,m}
		\end{array}\right)
		-\omega^2c_{3,n}\left(\begin{array}{c}
			[1-2\delta(n+2)]\tau^2c_{1,n,m}+(2n+3)c_{2,n,m}	 \\
			(2n-1)	\delta\tau^2 c_{1,n,m}-[2(n-1)+\delta]c_{2,n,m}
		\end{array}\right)
		+\mathcal{O}\left(\frac{\omega^4}{n^4}\left(\begin{array}{c}
			c_{1,n,m}	 \\
			c_{2,n,m}
		\end{array}\right)\right),\notag
		\end{align}
		where
		\begin{align}
			c_{1,n,m}&=\frac{(2n+1)f_{n,m}\mu[-(2n+1)j_n(k_s)+k_sj_{n+1}(k_s)]}{\delta(n+2)+n-1},\notag\\
			c_{2,n,m}&=\frac{(2n+1)f_{n,m}\mu[\delta k_s j_{n+1}(k_s)+(n-1)(1-\delta)j_n(k_s)]}{\delta(n+2)+(n-1)},\notag\\
			c_{3,n}&=\frac{\rho}{[\delta(n+2)+n-1](2n+3)(-2n+1)\mu}.\notag
		\end{align}
		By substituting \eqref{eq:jn expansion} and \eqref{eq:hn qiudao} into \eqref{eq:x_nm}, we get 
		\begin{align}\label{eq:X solution}
			\boldsymbol{x}_{n,m}
			=\frac{f_{n,m}(2n+1)\rho^{\frac{n}{2}}\omega^n}{(2n-1)!![\delta(n+2)+n-1]\mu^{\frac{n-2}{2}}}\cdot\left(\begin{array}{c}
				-1	 \\
				\frac{(n-1)(1-\delta)}{2n+1}
			\end{array}\right)
			+\mathcal{O}\left(\begin{array}{c}
				{x}_{1,n,m,\omega}	 \\
				{x}_{2,n,m,\omega}
			\end{array}\right).
		\end{align}
		with 
		\begin{align}
			{x}_{1,n,m,\omega}=\frac{f_{n,m}\rho^{\frac{n+2}{2}}\omega^{n+2}}{(2n-1)!![\delta(n+2)+n-1]\mu^{\frac{n}{2}}}, 
			\quad {x}_{2,n,m,\omega}=\frac{f_{n,m}[(1-n)+\delta(n+1)]\rho^{\frac{n+2}{2}}\omega^{n+2}}{(2n+3)(2n-1)!![\delta(n+2)+n-1]\mu^{\frac{n}{2}}},\notag
		\end{align}
		where $\omega\ll 1 $ is the fixed frequency of the incident wave defined in \eqref{eq:incident wave in T_n^m}, $\lambda,\mu,\rho$ satisfy the conditions of $\eqref{eq:lambda mu rho O1}$. Combining $\eqref{eq:B X b defination}$ and $\eqref{eq:X solution}$, it directly yields that
		\begin{align}
			\boldsymbol{\varphi_1}
			&=\sum_{m=-n}^n \frac{-f_{n,m}(2n+1)\rho^{\frac{n}{2}}\omega^n}{(2n-1)!![\delta(n+2)+n-1]\mu^{\frac{n-2}{2}}}
			\left(1+\mathcal{O}\left( -\frac{\rho\omega^2}{(2n+1)\mu}\right)\right)\mathcal{T}_n^m,\label{eq:varphi 1  definition of rho} \\
			\boldsymbol{\varphi_2} 
			&=\sum_{m=-n}^n\frac{f_{n,m}(n-1)(1-\delta)\rho^{\frac{n}{2}}\omega^n}{(2n-1)!![\delta(n+2)+n-1]\mu^{\frac{n-2}{2}}}\left(1+\mathcal{O}\left(\frac{[(1-n)+\delta(n+1)]\rho\omega^2}{(2n+3)(n-1)(1-\delta)\mu}\right)\right)\mathcal{T}_n^m.	\label{eq:varphi 2  definition of rho}
		\end{align}
	%
	%
	By virtue of \eqref{eq:total field}, \eqref{eq:S Tnm in ball} and \eqref{eq:varphi 1  definition of rho}, we obtain \eqref{eq:u single potential inside the unit sphere}. 
	Similarly, using formulas \eqref{eq:total field}, \eqref{eq:S partial outside ball} and \eqref{eq:varphi 2  definition of rho}, one can derive \eqref{eq:u^s expansion}.

	The proof is complete.
\end{proof}

In Theorem~\ref{thm:internal surface localization for scattering problem}, under the assumptions \eqref{eq:lambda mu rho}--\eqref{eq:assp omega}, we rigorously prove that both the interior total field \( \mathbf{u}|_D \) and the exterior scattered field \( \mathbf{u}^s|_{\mathbb{R}^3 \setminus \overline{D}} \) of \eqref{eq:system} are boundary localized in the sense of Definition~\ref{def:surface localized}, provided that the incident wave \( \mathbf{u}^i \) is chosen as in \eqref{eq:incident wave in T_n^m} with a sufficiently large \( n \).  For a fixed material parameter $\delta$, where $\delta$ denotes the ratio of the Lam\'e parameters between the hard elastic inclusion $D$ and the soft elastic background $\mathbb{R}^3 \setminus D$ as defined in \eqref{eq:lambda mu rho}, and for a prescribed sufficiently small boundary localization level $\varepsilon$, we can select the index $n$ of the incident wave defined in \eqref{eq:incident wave in T_n^m} to achieve boundary localization in both the interior total field $\mathbf{u}|_D$ and the exterior scattered field $\mathbf{u}^s|_{\mathbb{R}^3 \setminus \overline{D}}$ near $\partial D$.

\begin{thm}\label{thm:internal surface localization for scattering problem}
	Consider the elastic scattering problem \eqref{eq:system} associated with the elastic inclusion $(D;\tilde{\lambda},\tilde{\mu},\tilde{\rho})$ embedded in a homogeneous elastic medium $(\mathbb R^3 \setminus \overline D;\lambda,\mu,\rho)$, where $D$ is a unit ball centered at the origin in $\mathbb R^3$.  Let $B_R$ denote a ball of radius $R$ centered at the origin in $\mathbb{R}^3$ with $D\subseteq B_R$. Recall the definitions of  $ \mathcal{M}_{-}^{\xi_1} $ and $ \mathcal{M}_{+}^{\xi_2} $ in \eqref{eq:Mdef}, expressed in terms of $ \xi_1 = 1 - \gamma_1 $ and $ \xi_2 = \gamma_2 - 1 $, where $ \gamma_1 \in (0, 1) $ and $ \gamma_2 \in (1, R) $ are constants.  Let the corresponding internal total wave field $\mathbf{u}$ in $D$ and external scattered field $\mathbf{u}^s$ in $\mathbb R^3 \setminus  D$ be the solution to \eqref{eq:system} associated with incident wave $\mathbf u^i$ defined in \eqref{eq:incident wave in T_n^m}. Denote $\lfloor \cdot \rfloor$ by the floor function. 
	
	Under the assumptions \eqref{eq:lambda mu rho}-\eqref{eq:assp omega}, for any $ \gamma_1\in (0,1)$, $\gamma_2\in (1,R)$ and sufficient small $\varepsilon \in \mathbb R_+$, if the index $n$ associated with $\mathbf u^i$ defined in \eqref{eq:incident wave in T_n^m} satisfies 
	\begin{align}\label{eq:n1n2 max}
		\mbox{	
	$n\geqslant\max\{n_1,n_2\}$},
	\end{align}
	where 
	\begin{equation}\label{eq:n1 n2 def}
		n_1=\left\lfloor\frac{1}{2}\left(\frac{\ln{\varepsilon}}{\ln \gamma_1}-3\right)\right\rfloor+1, \qquad n_2=\left\lfloor\frac{1}{2}\left(1-\frac{\ln \varepsilon}{\ln \gamma_2}\right)\right\rfloor+1,
	\end{equation}
	then 
	\begin{equation}\label{eq:thm 3.1}
		\frac{\|\mathbf{u}\|_{L^2 (D\setminus \mathcal{M}_{-}^{1-\gamma_1}(\partial D))^3}^2}{\| \mathbf{u}\|_{L^2(D)^3}^2}\leqslant \varepsilon+\mathcal{O}\left(\varepsilon\omega^2\right) \ll 1, \qquad 
		 \frac{\|\mathbf{u}^s\|_{L^2 ((B_R \setminus \overline D) \setminus \mathcal{M}_{+}^{\gamma_{2}-1}(\partial D))^3}^2}{\|	\mathbf{u}^s\|_{L^2(B_R \setminus \overline D)^3}^2}\leqslant \varepsilon+\mathcal{O}\left(\varepsilon\omega^2\right) \ll 1.
	\end{equation}
\end{thm}

\begin{proof}

	Firstly, we shall derive the asymptotic analysis for  the $L^2$ norm of the internal total field $\mathbf{u}|_D$ within $D\setminus \mathcal{M}_{-}^{1-\gamma_1}(\partial D)$ with respect to the parameter $\omega$.
	Since $\mathbf u^i$ is given by \eqref{eq:incident wave in T_n^m}, according to Lemma \ref{lem:total field and scatter field}, we know that the explicit expression of ${\bf u}|_{D}$ is given by \eqref{eq:u single potential inside the unit sphere}. By the orthogonality of $\mathcal{T}_n^m$  for different $n $ and $m$, under the assumptions \eqref{eq:lambda mu rho}-\eqref{eq:assp omega} we have
	\begin{align}\label{eq:u L2norm}
		&\| {\bf u}\|_{L^2(D\setminus \mathcal{M}_{-}^{1-\gamma_1}(\partial D))^3}^2\notag\\
		=&\int_{D\setminus \mathcal{M}_{-}^{1-\gamma_1}(\partial D)}\left| \sum_{m=-n}^{n}\left[\frac{f_{n,m}\delta\rho^{\frac{n}{2}}|\bmf x|^n\omega^n}{(2n-1)!![\delta(n+2)+n-1]\mu^{\frac{n}{2}}}\left(1+\mathcal{O}\left(\frac{\rho\omega^2}{(2n+1)\mu}\right)\right)\right] \mathcal{T}_n^m \right|^2 \rmd \mathrm{x} \notag\\
		=&\int_0^{2 \pi} \int_0^\pi \int_0^{\gamma_1}  \sum_{m=-n}^{n}\left[\frac{f_{n,m}\delta\rho^{\frac{n}{2}}r^{n}\sin\theta\omega^n}{(2n-1)!![\delta(n+2)+n-1]\mu^{\frac{n}{2}}}\left(1+\mathcal{O}\left(\frac{\rho\omega^2}{(2n+1)\mu}\right)\right)\right]^2 \cdot r^2 \sin\theta \notag\\
		&\quad\times n(n+1)\rmd r \rmd \theta \rmd \varphi \notag\\
		=&\sum_{m=-n}^n  \frac{4\pi f_{n,m}^2  n(n+1) \omega^{2 n} \delta^2 \rho^n}{[(2 n-1) ! !]^2[\delta(n+2)+n-1]^2 \mu^n} \int_0^{\gamma_1} r^{2 n+2} \rmd  r+\mathcal{O}\left(K_{n,\delta,\rho,\mu}^{\prime}\gamma_{1}^{2n+3}\omega^{2 n+2} \right)\notag\\
		=&K_{n,\delta,\rho,\mu} \gamma_{1}^{2n+3}\omega^{2 n} +\mathcal{O}\left(K_{n,\delta,\rho,\mu}^{\prime}\gamma_{1}^{2n+3}\omega^{2 n+2}\right)\notag \\
		=&K_{n,\delta,\rho,\mu} \omega^{2 n} \gamma_1^{2 n+3}\left(1+\mathcal{O}\left(\omega^2\right)\right), 
	\end{align}
	where  
	\begin{align}
		K_{n,\delta,\rho,\mu}&=\frac{4 \pi n(n+1) \rho^n\delta^2}{(2 n+3) [(2 n-1) ! !]^2[\delta(n+2)+n-1]^2 \mu^n}\sum_{m=-n}^n  |f_{n,m}|^2 , \notag\\ 
		K_{n,\delta,\rho,\mu}^{\prime}&=\frac{4 \pi n(n+1)\rho^{n+1}\delta^2}{(2n+3)!!(2n-1)!![\delta(n+2)+n-1]^2\mu^{n+1}}\sum_{m=-n}^n |f_{n,m}|^2. \notag 
	\end{align}
	Using a  similar argument for deriving \eqref{eq:u L2norm}, it directly follows that 
	$$
	\|\mathbf{u}\|_{L^2\left(D\right)^3}^2=K_{n,\delta,\rho,\mu} \omega^{2 n} \left(1+\mathcal{O}\left(\omega^2\right)\right).
	$$ 
	Due to $n\geqslant n_1 $, where $n_1$ is given by \eqref{eq:n1 n2 def}, we know that $\gamma_1^{2n+3} \leqslant \varepsilon $.  Therefore, one has 
	\begin{align}
		\frac{\|\mathbf{u}\|_{L^2(D\setminus \mathcal{M}_{-}^{1-\gamma_1}(\partial D))^3}^2}{\|\mathbf{u}\|_{L^2\left(D\right)^3}^2}
		&=\gamma_1^{2 n+3}\left(1+\mathcal{O}\left(\omega^2\right)\right)\left(1-\mathcal{O}\left(\omega^2\right)\right) \leqslant \varepsilon+\mathcal{O}\left(\varepsilon\omega^2\right).\notag
	\end{align}

	Subsequently, we show that the external scattered field 
	$\mathbf{u}^s$ is also concentrated  on the boundary outside the inclusion $D$ in the sense of Definition \ref{def:surface localized}. Utilizing the similar argument for \eqref{eq:u L2norm}, in view of \eqref{eq:u single potential inside the unit sphere} in Lemma \ref{lem:total field and scatter field},   we have 
	\begin{align*}
		&\|\mathbf{u}^s\|_{L^2((B_R \setminus \overline D) \setminus \mathcal{M}_{+}^{\gamma_{2}-1}(\partial D))^3}^2\\
		=&\int_{(B_R \setminus \overline D) \setminus \mathcal{M}_{+}^{\gamma_{2}-1}(\partial D)} \left|\sum_{m=-n}^n \left[\frac{-f_{n,m}(n-1)(1-\delta) \rho^{\frac{n}{2}}\omega^n }{(2 n+1) ! ![\delta(n+2)+n-1]\mu^{\frac{n}{2}}|\bmf x|^{n+1} }\left(1+\mathcal{\mathcal{O}}\left(\frac{\omega^2}{n}\right)\right)\right]\mathcal{T}_n^m\right|^2\rmd \mathrm{x}\\
		=& \sum_{m=-n}^n \frac{4\pi f_{n,m}^2 n (n+1) (n-1)^2(1-\delta)^2 \rho^n \omega^{2 n}}{[(2 n+1)! !]^2[\delta(n+2)+n-1]^2 \mu^n} \int_{\gamma_2}^{R} r^{-2 n} \rmd r
		+\mathcal{O}\left(\frac{G_{n,\delta,\rho,\mu}^{\prime}\left(R^{2n-1}-{\gamma_2}^{2 n-1}\right)}{{\gamma_2}^{2 n-1} R^{2n-1}} \omega^{2 n+2}\right)\\
		=&\frac{G_{n,\delta,\rho,\mu}\left(R^{2n-1}-{\gamma_2}^{2 n-1}\right)}{{\gamma_2}^{2 n-1} R^{2n-1}} \omega^{2 n}\left(1+\mathcal{O}
		\left( \omega^{2}\right)\right),
	\end{align*}
	where
	\begin{align}
		G_{n,\delta,\rho,\mu}&=\frac{4\pi n(n+1)(n-1)^2(1-\delta)^{2}\rho^{n}}{[(2n+1)!!]^{2}[\delta(n+2)+n-1]^{2}(2n-1)\mu^{n}}\sum_{m=-n}^n |f_{n,m}|^2,\notag\\
		G_{n,\delta,\rho,\mu}^{\prime}&= \frac{4\pi n(n+1)[(1-n)+\delta(n+1)]^2\rho^{\frac{n+2}{2}}}{(2n+3)!!(2n+1)!![\delta(n+2)+n-1]^2(2n-1)\mu^{\frac{n+2}{2}}} \sum_{m=-n}^n |f_{n,m}|^2. \notag
	\end{align}
	Using a  similar argument for deriving \eqref{eq:u L2norm}, it is readily evident that 
	$$
	\| \mathbf{u}^s\|^2_{L^2\left(B_R \setminus \overline D\right)^3}= \frac{G_{n,\delta,\rho,\mu}\left(R^{2n-1}-1\right)}{ R^{2n-1}} \omega^{2 n}\left(1+\mathcal{O}\left(\omega^{2}\right)\right).
	$$
	Due to $n\geqslant n_2 $, where $n_2$ is given by \eqref{eq:n1 n2 def}, we know that $\frac{1}{\gamma_2^{2 n-1}}<\varepsilon$.  Therefore, it yields that
	\begin{align}
		\frac{\| \mathbf{u}^s\|_{L^2((B_R \setminus \overline D) \setminus \mathcal{M}_{+}^{\gamma_{2}-1}(\partial D))^3}^2}{\| \mathbf{u}^s\|^2_{L^2\left(B_R \setminus \overline D\right)^3}}
		=&\frac{\left(R^{2n-1}-{\gamma_2}^{2 n-1}\right)}{\gamma_2^{2n-1}\left(R^{2n-1}-1\right)}\left(1+\mathcal{O}
		\left(\omega^{2}\right)\right)\left(1-\mathcal{O}
		\left( \omega^{2}\right)\right) \notag\\
		=&\frac{1-\left(\frac{\gamma_{2}}{R}\right)^{2n-1}}{\gamma_2^{2n-1}\left(1-\frac{1}{R^{2n-1}}\right)}\left(1+\mathcal{O}
		\left(\omega^{2}\right)\right)\left(1-\mathcal{O}
		\left( \omega^{2}\right)\right) \notag\\
		=&\frac{1}{\gamma_2^{2n-1}}\left(\mathcal{O}\left(1\right)+\mathcal{O}
		\left(\omega^{2}\right)\right)\left(1-\mathcal{O}
		\left( \omega^{2}\right)\right)	\leqslant \varepsilon +\mathcal{O}\left(\varepsilon\omega^2\right). \notag
	\end{align}

	The proof is complete.
\end{proof}

	\begin{rem}

		In Theorem \ref{thm:internal surface localization for scattering problem}, the total wave field generated within the inclusion \( D \) and the exterior scattered wave field outside \( D \) can be localized at the boundary, provided that the incident wave \( \mathbf{u}^i \) takes the form specified in \eqref{eq:incident wave in T_n^m} and \( D \) is a unit ball. Specifically, this result holds for inclusions of radial geometry, a structure frequently encountered in metamaterial science \cite{LAI,LIUZHANG}. Such geometries are often employed to achieve effective negative material parameters, making them a subject of significant interest in the field. This theorem represents the first rigorous attempt to characterize the physical wave patterns associated with high contrast radial inclusions using mathematical analysis. From Theorem~\ref{thm:internal surface localization for scattering problem}, it is evident that both the internal wave and the scattered wave propagate along the boundary of the hard inclusion \( D \) when the incident wave \( \mathbf{u}^i \) is chosen as in \eqref{eq:incident wave in T_n^m}. This type of wave propagation is consistent with phenomena observed in the literature for hard inclusions in material science \cite{LAI,LIUZHANG}. We believe that the findings of this theorem extend to high contrast inclusions of more general shapes. However, a detailed investigation of such cases is beyond the scope of this work and will be addressed in future studies.
	\end{rem}

		\begin{rem}  
			From \eqref{eq:thm 3.1}, it follows directly that  
			\begin{align}\notag
				\frac{\|\mathbf{u}\|_{L^2 ( \mathcal{M}_{-}^{1-\gamma_1}(\partial D))^3}^2}{\| \mathbf{u}\|_{L^2(D)^3}^2} = 1 - \mathcal{O}(\varepsilon), \qquad  
				\frac{\|\mathbf{u}^s\|_{L^2 ( \mathcal{M}_{+}^{\gamma_{2}-1}(\partial D))^3}^2}{\| \mathbf{u}^s\|_{L^2(B_R \setminus \overline{D})^3}^2} = 1 - \mathcal{O}(\varepsilon).  
			\end{align}  
			Here, $\mathcal{M}_{-}^{1-\gamma_1}(\partial D)$ and $\mathcal{M}_{+}^{\gamma_{2}-1}(\partial D) $, defined in \eqref{eq:Mdef} with $ \gamma_1 \in (0, 1) $ and $ \gamma_2 \in (1, R) $ being constants. For a given boundary localization level $\mathcal{O}\left(\varepsilon\right)$ and an index $n$ of the incident wave defined in \eqref{eq:incident wave in T_n^m} that satisfies \eqref{eq:n1n2 max}, where $n_1$ and $n_2$ are defined as in \eqref{eq:n1 n2 def}, the internal total field $\mathbf{u}|_D$ and the external scattered field $\mathbf{u}^s|_{B_R \setminus \overline{D}}$ exhibit boundary localization within $\mathcal{M}_{-}^{1-\gamma_1}(\partial D)$ and $\mathcal{M}_{+}^{\gamma_{2}-1}(\partial D)$, respectively. This observation is consistent with the boundary localization framework introduced in  \eqref{eq:220} of Definition \ref{def:surface localized}.

		\end{rem}

Theorem \ref{thm:internal surface localization for scattering problem} demonstrates that for any fixed material parameter $\delta$ and prescribed boundary localization levels $\varepsilon$, selecting the index $n$ of the incident wave $\mathbf{u}^i$ to satisfy \eqref{eq:n1n2 max} ensures that both the interior total field $\mathbf{u}|_D$ and the exterior scattered field $\mathbf{u}^s|_{\mathbb{R}^3 \setminus \overline{D}}$ exhibit  boundary localization. Corollary \ref{cor:3.2} further shows that,  for a fixed boundary localization level $\varepsilon$, under the condition \eqref{eq:n27}  on  the index $n$ of the incident wave $\mathbf{u}^i$,  the boundary localization for the interior total field $\mathbf{u}|_D$ and the exterior scattered field $\mathbf{u}^s|_{\mathbb{R}^3 \setminus \overline{D}}$  can be achieved through appropriately tuning the material parameter $\delta$ with respect to $\varepsilon$, $\gamma_1$ and $\gamma_2$.

\begin{cor}\label{cor:3.2}
	Under the same assumptions as Theorem \ref{thm:internal surface localization for scattering problem}, for fixed parameters $\gamma_1 \in (0,1)$ and $\gamma_2 \in (1,R)$, and a fixed boundary localization level $\varepsilon$ defined in Definition \ref{def:surface localized} with $\varepsilon<\gamma_{1}$, if the material parameter $\delta$ satisfies
	\begin{align}\label{eq:327}
		\delta \leqslant \beta =\min\left\{\frac{2\ln \gamma_1}{\ln\varepsilon - \ln \gamma_1}, \frac{2\ln \gamma_2}{3\ln{\gamma_2} - \ln \varepsilon} \right\} ,
	\end{align}
	and if the index $n$ associated with the incident wave $\mathbf{u}^i$ defined in \eqref{eq:incident wave in T_n^m} satisfies
	\begin{align}\label{eq:n27}
		n \geqslant \frac{1}{\delta},
	\end{align}
	where $\delta$ represents the ratio of the Lam\'e parameters between the hard elastic inclusion $D$ and the soft elastic background $\mathbb{R}^3 \setminus D$, then the interior total field $\mathbf{u}|_D$ and the exterior scattered field $\mathbf{u}^s|_{\mathbb{R}^3 \setminus \overline{D}}$ satisfy the boundary localization property \eqref{eq:thm 3.1}.
\end{cor}

\begin{proof}
	When \eqref{eq:327} is satisfied, since  $n_1$ and $n_2$ are as defined in \eqref{eq:n1 n2 def}, it can be verified that 
	\begin{align}\label{eq:326}
		\frac{1}{\delta} \geqslant \max\{n_1, n_2\}. 
	\end{align}
	where $n_1$ and $n_2$ are as defined in \eqref{eq:n1 n2 def}. Combining \eqref{eq:n27} and \eqref{eq:326}, we deduce that the index $n$ satisfies the condition \eqref{eq:n1n2 max} in Theorem \ref{thm:internal surface localization for scattering problem}. Consequently, Corollary \ref{cor:3.2} can be proved directly from  Theorem \ref{thm:internal surface localization for scattering problem}. 
\end{proof}

	\section{Surface resonance and stress concentration}\label{sec:summary of major findings}

In Theorem \ref{thm:internal surface localization for scattering problem} the boundary localization of the internal total field $\mathbf{u}|_D$ and the external scattered field $\mathbf{u}^s|_{\mathbb{R}^3 \setminus \overline{D}}$ is associated with the incident wave $\mathbf{u}^i$ given by~\eqref{eq:incident wave in T_n^m} in the sub-wavelength regime. In this section we shall further show the corresponding surface resonance occurs.  To simplify the exposition, throughout this section, we primarily consider a special form of $\mathbf{u}^i$ introduced by \eqref{eq:incident wave in T_n^m} as follows:  
	\begin{align}\label{eq:u^i 4.1}  
		\mathbf{u}^{i}(\bmf x)= f_{n, n} j_n\left(k_s|\mathbf{x}|\right) \mathcal{T}_n^n (\theta, \varphi),  
	\end{align}  
	where $ f_{n,n} \in \mathbb{C} $ denotes a non-zero constant. In Theorem~\ref{thm:nabla u in thm}, we demonstrate that the internal total field $\mathbf{u}|_D$ and the external scattered field $\mathbf{u}^s|_{\mathbb{R}^3 \setminus \overline{D}}$ exhibit surface resonance, as defined in Definition~\ref{def:surface resonant}. According to Definition~\ref{def:quasi minnaert}, the high contrast inclusion $D$ constitutes a quasi-Minnaert resonator. Moreover, Theorem~\ref{thm:Eu definition in thm} establishes that the fields $\mathbf{u}|_D$ and $\mathbf{u}^s|_{\mathbb{R}^3 \setminus \overline{D}}$ of the scattering problem \eqref{eq:system}, associated with $\mathbf u^i$ given by \eqref{eq:u^i 4.1} and $D$, generate strong stress concentrations near $\partial D$. If the more general form \eqref{eq:incident wave in T_n^m} of $\mathbf{u}^i$ is considered, the proofs of Theorems $\ref{thm:nabla u in thm}$ and $\ref{thm:Eu definition in thm}$ become significantly more involved due to complex technical arguments; therefore, we restrict our analysis to $\mathbf{u}^i$ in the special form \eqref{eq:u^i 4.1}.

   In Theorem~\ref{thm:nabla u in thm}, we establish that for a fixed material parameter $\delta$ defined as the ratio of the Lam\'e parameters of the hard inclusion to those of the soft background medium, and a sufficiently small $\varepsilon$ representing the level of boundary localization, an appropriate index $n$ of the incident wave, as defined in \eqref{eq:u^i 4.1}, can be selected to induce surface resonance.

	\begin{thm}\label{thm:nabla u in thm}  
		Consider the elastic scattering problem~\eqref{eq:system}, where the unit ball $ (D; \tilde{\lambda}, \tilde{\mu}, \tilde{\rho}) $ is embedded within a homogeneous elastic medium $ (\mathbb{R}^3 \setminus \overline{D}; \lambda, \mu, \rho) $. Under the assumptions~\eqref{eq:lambda mu rho}--\eqref{eq:assp omega}, the incident wave $\mathbf{u}^i$ is chosen as defined in~\eqref{eq:u^i 4.1}. Recall that  $\mathcal{M}_{-}^{1-\gamma_1}(\partial D)$ and $\mathcal{M}_{+}^{\gamma_2-1}(\partial D)$, are introduced in~\eqref{eq:Mdef}  with $\gamma_1\in(0,1)$ and $\gamma_2\in(1,R)$, respectively. For a fixed sufficiently small boundary localization level $\varepsilon \ll 1$, if the index $n$ of the incident wave defined in \eqref{eq:u^i 4.1} satisfies 
		\begin{align}\label{eq:n4.2}
			n \geqslant \max\{ n_1,n_2, 1/\delta^2\}
		\end{align}
		 where $\delta$ represents the ratio of the Lam\'e parameters between the hard elastic inclusion $D$ and the soft elastic background $\mathbb{R}^3 \setminus D$ defined in \eqref{eq:lambda mu rho} and $n_1$, $n_2$ are defined in \eqref{eq:n1 n2 def}, then the corresponding total field $\mathbf{u}$ in $D$ and the scattered field $\mathbf{u}^s$ in $\mathbb{R}^3 \setminus D$ satisfy: 
		\begin{align}\label{eq:nablau nabla us gg 1} 
			\frac{\|\nabla \bmf u\|_{L^2(\mathcal{M}_{-}^{1-\gamma_1}(\partial D))^3}}{\| \bmf u^i\|_{L^2(D)^3}}
			&\geqslant \frac{n \delta}{4\sqrt{\pi}} \gg 1,
			\qquad
			\frac{\|\nabla \bmf u^s\|_{L^2(\mathcal{M}_{+}^{\gamma_2-1}(\partial D))^3}}{\| \bmf u^i\|_{L^2(D)^3}}
			\geqslant  \frac{n}{9 \sqrt{\pi} } \gg 1. 
		\end{align}
		Namely, the surface resonance occurs. 
	\end{thm}

	\begin{rem}\label{rem:4.1}

	Due to the boundary localization level $\varepsilon \ll 1$, the definitions of $n_1$ and $n_2$ in \eqref{eq:n1 n2 def} imply
\begin{align}\label{eq:n44}
    n \geqslant \max\{n_1, n_2\} \gg 1.
\end{align}
For parameters $\gamma_1 \in (0,1)$ and $\gamma_2 \in (1,R)$ defining the boundary layers $D \setminus \mathcal{M}_{-}^{1-\gamma_1}(\partial D)$ and $(B_R \setminus \overline{D}) \setminus \mathcal{M}_{+}^{\gamma_2-1}(\partial D)$, if $|\gamma_j - 1| \ll \epsilon$ ($j=1,2$) for a sufficiently small $\epsilon \in \mathbb{R}_+$, then even with fixed $\varepsilon$, the constants $n_j$ can satisfy $n_j \gg 1$. From \eqref{eq:n4.2}, it follows that $\sqrt{n}\delta = \mathcal{O}(1)$, leading to $n\delta = \mathcal{O}(\sqrt{n}) \gg 1$ by noting the aforementioned discussion. This verifies the first inequality in \eqref{eq:nablau nabla us gg 1}; the second inequality is then deduced from \eqref{eq:n44}. Thus, boundary localization and surface resonance are rigorously confirmed for $\mathbf{u}|_D$ and $\mathbf{u}^s|_{\mathbb{R}^3 \setminus D}$, respectively, as per Definitions~\ref{def:surface localized} and~\ref{def:surface resonant}. This establishes the intrinsic connection between quasi-Minnaert resonance and the incident wave field $\mathbf{u}^i$ in \eqref{eq:u^i 4.1} within the sub-wavelength regime.

\end{rem}

Now we are in the position to give the proof of Theorem \ref{thm:nabla u in thm}. 
\begin{proof}

In what follows, the proof will be divided into two parts.
		
		\medskip
		
		\noindent{\bf Part 1.} In this part, we shall prove $\frac{\|\nabla \bmf u\|_{L^2(\mathcal{M}_{-}^{1-\gamma_1}(\partial D))^3}}{\| \bmf u^i\|_{L^2(D)^3}}\geqslant \frac{n \delta}{4\sqrt{\pi}}$. 
		We first need to derive the  asymptotic analysis for $\nabla \bmf u|_{\mathcal{M}_{-}^{1-\gamma_1}(\partial D)}$ with respect to $\omega$. It turns out that we need calculate   $\nabla\left(j_n(\tilde{k}_s r) \mathcal{T}_n^n(\theta, \varphi)\right)$, where $\tilde{k}_s$ and $\mathcal{T}_n^n(\theta, \varphi)$ are defined by \eqref{eq:tilde ks relationship with ks} and \eqref{eq:Tnm definition}, respectively. In the following, let us derive the explicit expression for $\nabla \mathcal{T}_n^n(\theta, \varphi)$. We rewrite $ \mathcal{T}_n^n(\theta, \varphi)$ in \eqref{eq:Tnm definition} as 
		\begin{align}\label{eq:Tnm expansion in theta and varphi}
			\mathcal{T}_n^n(\theta, \varphi)
			=\frac{\rmi n}{\sin \theta} Y_n^n(\theta, \varphi) \hat{\theta}
			-C_n^n e^{\rmi n \varphi} \frac{\partial P_n^{n}(\cos\theta)}{\partial \theta}  \hat{\varphi}
			=A_{\hat{\theta}(n,\varphi,\theta)} \hat{\theta}+A_{\hat{\varphi}(n,\varphi,\theta)}\hat{\varphi},
		\end{align}
		where
		\begin{align}\label{eq:r hat theta hat varphi hat definition}
			\hat{r}=\left(\begin{array}{c}
				\sin \theta \cos \varphi \\
				\sin \theta \sin \varphi \\
				\cos \theta
			\end{array}\right), \hat{\theta}=\left(\begin{array}{c}
				\cos \theta \cos \varphi \\
				\cos \theta \sin \varphi \\
				-\sin \theta
			\end{array}\right), \hat{\varphi}=\left(\begin{array}{c}
				-\sin \varphi \\
				\cos \varphi \\
				0
			\end{array}\right).
		\end{align}
		and 
		\begin{equation}\label{eq:A_theta A_varphi}
			A_{\hat{\theta}(n,\varphi,\theta)}=\frac{\rmi n}{\sin \theta} Y_n^n(\theta, \varphi),       \quad A_{\hat{\varphi}(n,\varphi,\theta)}=-C_n^n e^{\rmi n \varphi} \frac{\partial P_n^{n}(\cos\theta)}{\partial \theta} .
		\end{equation}
		Furthermore, by direct calculations, it is readily known that
		\begin{align}\label{eq:grad Tnm}
			\nabla \mathcal{T}_n^n(\theta, \varphi)
			=&-\frac{A_{\hat{\theta}(n,\varphi,\theta)}}{r} \hat{r} \otimes \hat{\theta}
			-\frac{A_{\hat{\varphi}(n,\varphi,\theta)}}{r} \hat{r} \otimes \hat{\varphi}
			+\frac{1}{r} \frac{\partial A_{\hat{\theta}(n,\varphi,\theta)}}{\partial \theta} \hat{\theta} \otimes 
			\hat{\theta} \\
			&+\left(\frac{1}{r \sin \theta} \frac{\partial A_{\hat{\theta}(n,\varphi,\theta)}}{\partial \varphi}-\cot \theta \frac{A_{\hat{\varphi}(n,\varphi,\theta)}}{r}\right) \hat{\theta} \otimes \hat{\varphi} \notag \\
			&+\frac{1}{r} \frac{\partial A_{\hat{\varphi}(n,\varphi,\theta)}}{\partial \theta} \hat{\varphi} \otimes \hat{\theta}+\left(\frac{1}{r \sin \theta} \frac{\partial A_{\hat{\varphi}(n,\varphi,\theta)}}{\partial \varphi}+\cot \theta \frac{A_{\hat{\theta}(n,\varphi,\theta)}}{r}\right) \hat{\varphi} \otimes \hat{\varphi}. \notag 
		\end{align}
		Combining \eqref{eq:Tnm expansion in theta and varphi} and \eqref{eq:grad Tnm}, we can directly derive 
		\begin{align}\label{eq:nabla jn Tnm}
			&\nabla\left(j_n(\tilde{k}_s r) \mathcal{T}_n^n(\theta, \varphi)\right)
			=j_n^{\prime}(\tilde{k}_s r) \tilde{k}_s\left(\begin{array}{c}
				\sin \theta \cos \varphi \\
				\sin \theta \sin \varphi \\
				\cos \theta
			\end{array}\right)\left(\mathcal{T}_n^n\right)^{\top} 
			+j_n(\tilde{k}_s r) \nabla \mathcal{T}_n^n(\theta, \varphi) \\
			=&  A_{\hat{\theta}(n,\varphi,\theta)}\left(j_{n-1}(\tilde{k}_s r) \tilde{k}_s-\frac{n+2}{r} j_n(\tilde{k}_s r)\right) \hat{r} \otimes \hat{\theta}+A_{\hat{\varphi}(n,\varphi,\theta)}\left(j_{n-1}(\tilde{k}_s r) \tilde{k}_s-\frac{n+2}{r} j_n(\tilde{k}_s r)\right) \notag\\
			&\times \hat{r} \otimes \hat{\varphi}+\left[\frac{1}{r} \frac{\partial A_{\hat{\theta}(n,\varphi,\theta)}}{\partial \theta} \hat{\theta} \otimes 
			\hat{\theta}
			+\left(\frac{1}{r \sin \theta} \frac{\partial A_{\hat{\theta}(n,\varphi,\theta)}}{\partial \varphi}-\cot \theta \frac{A_{\hat{\varphi}(n,\varphi,\theta)}}{r}\right) \hat{\theta} \otimes \hat{\varphi}
			\right.\notag\\
			&\left.+\frac{1}{r} \frac{\partial A_{\hat{\varphi}(n,\varphi,\theta)}}{\partial \theta} \hat{\varphi} \otimes \hat{\theta}+\left(\frac{1}{r \sin \theta} \frac{\partial A_{\hat{\varphi}(n,\varphi,\theta)}}{\partial \varphi}+\cot \theta \frac{A_{\hat{\theta}(n,\varphi,\theta)}}{r}\right) \hat{\varphi} \otimes \hat{\varphi}\right]\cdot j_n\left(\tilde{k}_s r\right).\notag
		\end{align}
		%
		Following a similar argument for proving Lemma \ref{lem:total field and scatter field}, by noting $\mathbf u^i$ given by  \eqref{eq:u^i 4.1},  using the integral system \eqref{eq:BX=b}, one can obtain the  asymptotic analysis for $\bmf u |_D$ as follows 
		\begin{align}
			\bmf u|_D =f_{n,n} \left[C_{n,\delta,\tau}+\mathcal{O}\left(  C_{n,\delta,\tau}^{\prime} \omega^2\right)\right] j_n(\tilde{k}_s|\bmf x|)\mathcal{T}_n^n(\theta, \varphi), \notag
		\end{align}
		where
		\begin{align}\label{eq:C_nm delta definition}
			C_{n,\delta,\tau}=\frac{(2 n+1) \delta}{\tau^n(\delta(n+2)+n-1)}, \quad C_{n,\delta,\tau}^{\prime}=\frac{\delta[(2n-1)+\tau^2(2n+1)]}{[\delta(n+2)+n-1](2n-1)}.
		\end{align}
		Here, $\delta $ is the comparison ratio  defined by \eqref{eq:lambda mu rho}, $\tau<1$ is also the given contrast defined in \eqref{eq:tau defination}. 
		Therefore, it arrives at
		\begin{equation}\label{eq:nabla u definition}
			\nabla \bmf u=  f_{n,n} \left[C_{n,\delta,\tau}+\mathcal{O}\left(  C_{n,\delta,\tau}^{\prime} \omega^2\right)\right]\nabla\left(j_n(\tilde{k}_s r) \mathcal{T}_n^n(\theta, \varphi)\right).
		\end{equation}
		Combining with \eqref{eq:nabla jn Tnm}, \eqref{eq:C_nm delta definition} and \eqref{eq:nabla u definition}, we derive that

		\begin{align}
			\|\nabla \bmf u\|_{L^2(\mathcal{M}_{-}^{1-\gamma_1}(\partial D))^3}^2
			=&\left[C_{n,\delta,\tau}+\mathcal{O}\left(  C_{n,\delta,\tau}^{\prime} \omega^2\right)\right]^2\int_{\mathcal{M}_{-}^{1-\gamma_1}(\partial D)} \left\{\left[j_{n-1}(\tilde{k}_s r) \tilde{k}_s r-({n+2}) j_n(\tilde{k}_s r)\right]^2 \right.\notag\\
			&\times P(n,\theta,\varphi)\left.+j_n^2(\tilde{k}_s r) Q(n,\theta,\varphi)\right\}\rmd  r \rmd \theta \rmd\varphi,\label{eq:nabla u expansion}
		\end{align}
		where
		\begin{align}
			P(n,\theta,\varphi)&=\left| f_{n,n} A_{\hat{\theta}(n,\varphi,\theta)}\right|^2 +\left| f_{n,n} A_{\hat{\varphi}(n,\varphi,\theta)}\right|^2\sin\theta,  \label{eq:P theta, varepsilon defination}\\
			Q(n,\theta,\varphi)&= \left| f_{n,n}\left(\frac{\partial A_{\hat{\theta}(n,\varphi,\theta)}}{\partial \theta}\right)\right|^2
			+\left| f_{n,n}\left(\frac{\partial A_{\hat{\varphi}(n,\varphi,\theta)}}{\partial \theta} \right)\right|^2  \sin\theta  + \frac{1}{\sin \theta}\left\{\left| f_{n,n}\left( \frac{\partial A_{\hat{\theta}(n,\varphi,\theta)}}{\partial \varphi}\right.\right.\right.\notag\\
			&  \left. \left. -\cos \theta A_{\hat{\varphi}(n,\varphi,\theta)}\right)\right|^2+ \left.\left| f_{n,n}\left(\frac{\partial A_{\hat{\varphi}(n,\varphi,\theta)}}{\partial \varphi}+\cos \theta A_{\hat{\theta}(n,\varphi,\theta)}\right)\right|^2\right\}.\label{eq:Q def}
		\end{align}
		Substituting \eqref{eq:ks kp defination}, \eqref{eq:cs cp defination},\eqref{eq:tilde ks relationship with ks}, \eqref{eq:jn expansion} and \eqref{eq:C1nC2nC3n} into \eqref{eq:nabla u expansion}, it holds that 
		\begin{align}\notag
			&\|\nabla \bmf u\|_{L^2(\mathcal{M}_{-}^{1-\gamma_1}(\partial D))^3}^2
			=C_{n,\delta,\tau}^2 C_1(n)\int_{\gamma_{1}}^{1}\frac{(n-1)^2(\tilde{k}_{s}r)^{2n}}{[(2n+1)!!]^2}\rmd r
			+C_{n,\delta,\tau}^2 \int_{\gamma_{1}}^{1}\frac{(\tilde{k}_s r)^{2n}}{[(2n+1)!!]^2}\rmd r \\
			&\times  C_2(n)
			+\mathcal{O}\left(\frac{C_{n,\delta,\tau}C_{n,\delta,\tau}^{\prime}\tau^{2n}C_2(n)\rho^n\omega^{2n+2}}{[(2n+1)!!]^2(2n+1)\mu^n}\right), \notag 
		\end{align}
		where
		\begin{align}
			C_1(n)=\int_{0}^{2\pi}\int_{0}^{\pi}P(n,\theta,\varphi)\rmd\theta \rmd\varphi ,\quad C_2(n)=\int_{0}^{2\pi}\int_{0}^{\pi}Q(n,\theta,\varphi)\rmd\theta\rmd\varphi. \notag
		\end{align}
		To facilitate subsequent analysis, we shall give the detailed proof of the asymptotic  behavior for $C_1(n)$ and $C_2(n)$ with respect to $n$  as follows 
		\begin{align}\label{eq:C1nC2nC3n}
			C_1(n)=n^2 \left| f_{n,n}\right|^2+\mathcal{O}\left(n \left| f_{n,n}\right|^2\right),\quad C_2(n)=n^4 \left| f_{n,n}\right|^2+\mathcal{O}\left(n^3 \left| f_{n,n}\right|^2\right), 
		\end{align}
		where $f_{n,n}\in \mathbb C$ is a constant defined in \eqref{eq:u^i 4.1}. By the definition of $C_1(n)$ and $C_2(n)$, using \eqref{eq:P theta, varepsilon defination} and \eqref{eq:Q def}, it yields that
		\begin{align} \label{eq:C1n=}
			C_1(n)&=\int_{0}^{2\pi}\int_{0}^{\pi}P(n,\theta,\varphi)\rmd\theta \rmd\varphi \notag\\
			&=\int_{0}^{2\pi}\int_{0}^{\pi}\left(\left| f_{n,n} A_{\hat{\theta}(n,\varphi,\theta)}\right|^2 +\left| f_{n,n} A_{\hat{\varphi}(n,\varphi,\theta)}\right|^2 \sin\theta\right)\rmd\theta \rmd\varphi \notag\\
			&=|f_{n,n}|^2 \int_{0}^{2\pi}\int_{0}^{\pi}
			\left(\left| A_{\hat{\theta}(n,\varphi,\theta)}\right|^2 
			+\left| A_{\hat{\varphi}(n,\varphi,\theta)}\right|^2 \sin\theta\right)\rmd\theta \rmd\varphi,
		\end{align}
		and 
		\begin{align}
			C_2(n)&=\int_{0}^{2\pi}\int_{0}^{\pi}Q(n,\theta,\varphi)\rmd\theta\rmd\varphi \label{eq:C2n=}\\
			&=\int_{0}^{2\pi}\int_{0}^{\pi}\left| f_{n,n}\left(\frac{\partial A_{\hat{\theta}(n,\varphi,\theta)}}{\partial \theta}\right)\right|^2
			+\left|f_{n,n}\left(\frac{\partial A_{\hat{\varphi}(n,\varphi,\theta)}}{\partial \theta} \right)\right|^2 \sin \theta + \frac{1}{\sin \theta}\left\{\left| f_{n,n}\right.\right.\notag\\
			& \left. \left. \times \left( \frac{\partial A_{\hat{\theta}(n,\varphi,\theta)}}{\partial \varphi}-\cos \theta A_{\hat{\varphi}(n,\varphi,\theta)}\right)\right|^2\right.+ \left.\left| f_{n,n}\left(\frac{\partial A_{\hat{\varphi}(n,\varphi,\theta)}}{\partial \varphi}+\cos \theta A_{\hat{\theta}(n,\varphi,\theta)}\right)\right|^2\right\} \rmd\theta\rmd\varphi , \notag\\
			&=|f_{n,n}|^2 \int_{0}^{2\pi}\int_{0}^{\pi}\left|\frac{\partial A_{\hat{\theta}(n,\varphi,\theta)}}{\partial \theta}\right|^2
			+\left|\frac{\partial A_{\hat{\varphi}(n,\varphi,\theta)}}{\partial \theta} \right|^2 \sin \theta + \frac{1}{\sin \theta}\left\{\left| \frac{\partial A_{\hat{\theta}(n,\varphi,\theta)}}{\partial \varphi}-\cos \theta A_{\hat{\varphi}(n,\varphi,\theta)}\right|^2\right.\notag\\
			&~ + \left.\left|\frac{\partial A_{\hat{\varphi}(n,\varphi,\theta)}}{\partial \varphi}+\cos \theta A_{\hat{\theta}(n,\varphi,\theta)}\right|^2\right\} \rmd\theta\rmd\varphi , \notag
		\end{align}
		where $f_{n,n}\in \mathbb C$ is a constant defined in \eqref{eq:u^i 4.1},  $A_{\hat{\theta}(n,\varphi,\theta)}$ and $A_{\hat{\varphi}(n,\varphi,\theta)}$ are given by  \eqref{eq:A_theta A_varphi}.  In order to prove \eqref{eq:C1nC2nC3n}, we first recall that 
		\begin{align}
			\frac{d P_n^{n}(\cos\theta)}{d\theta}=n P_n^{n-1}(\cos\theta), \quad 	\frac{P_n^{n}(\cos\theta)}{\sin\theta}=-(2n-1)P_n^{n-1}(\cos\theta).\label{eq:dPnm} 
		\end{align}
		According to $C_n^n$ and $P_n^{n}(\cos\theta)$ defined in \eqref{eq:ynm and cnm def}, substituting \eqref{eq:dPnm} and \eqref{eq:ynm and cnm def} into \eqref{eq:A_theta A_varphi}, it can be derived  that
		\begin{align}
			&\int_{0}^{2\pi}\int_{0}^{\pi}\left|A_{\hat{\theta}(n,\varphi,\theta)}\right|^2 \rmd\theta\rmd\varphi 
			=\int_{0}^{2\pi}\int_{0}^{\pi}\left[(2n-1)^2 n^2\left|P_{n-1}^{n-1}\left(\cos\theta\right)\right|^2\right]\left(C_n^n\right)^2 \rmd\theta\rmd\varphi \notag\\
			&=\frac{(2n-1)^2 n^2 \left(C_n^n\right)^2 }{\left(C_{n-1}^{n-1}\right)^2}\int_{0}^{2\pi}\int_{0}^{\pi} |Y_{n-1}^{n-1}|^2 \rmd\theta\rmd\varphi =n^2+\frac{n}{2}=n^2+\mathcal{O}\left(n\right). \label{eq:Atheta}
		\end{align}
		By employing the same proof as in \eqref{eq:Atheta} and combining \eqref{eq:ynm and cnm def}, \eqref{eq:A_theta A_varphi}, and \eqref{eq:dPnm}, it follows that
		\begin{align}
			&\int_{0}^{2\pi}\int_{0}^{\pi} \left|A_{\hat{\varphi}(n,\varphi,\theta)}\right|^2 \rmd\theta\rmd\varphi 
			= \int_{0}^{2\pi}\int_{0}^{\pi} \left[ n^2 \left| P_n^{n-1}(\cos\theta)\right|^2\right]\left(C_n^n\right)^2 \rmd\theta\rmd\varphi =\frac{n}{2},\label{eq:Avarphi}\\
			&\int_{0}^{2\pi}\int_{0}^{\pi} \left|\frac{\partial A_{\hat{\theta}(n,\varphi,\theta)}}{\partial \theta}\right|^2 \rmd\theta\rmd\varphi 
			=\int_{0}^{2\pi}\int_{0}^{\pi}(C_n^n)^2 \left[(2n-1)n(n-1)\right]^2 |P_{n-1}^{n-2}(\cos\theta)| \rmd\theta\rmd\varphi \notag\\
			&\quad=\frac{1}{2}n^{3}-\frac{3}{4}n^{2}+\frac{1}{4}n,\label{eq:partial Atheta}\\
			&\int_{0}^{2\pi}\int_{0}^{\pi}  \left|\frac{\partial A_{\hat{\varphi}(n,\varphi,\theta)}}{\partial \theta}\right|^2 \rmd\theta\rmd\varphi 
			=\int_{0}^{2\pi}\int_{0}^{\pi} \frac{\left(C_n^n\right)^2 n^2}{4}\left[(2n-1)P_n^{n-2}(\cos\theta)-P_n^{n}(\cos\theta)\right]^2 \rmd\theta\rmd\varphi  \notag\\
			&\quad=\frac{3}{8} n^2-\frac{1}{16}n. \label{eq:partial Avarphi}
		\end{align}

		Utilizing the similar argument  and combining with \eqref{eq:ynm and cnm def}, \eqref{eq:A_theta A_varphi} and \eqref{eq:dPnm}, one can obtain that
		\begin{align}
			&\int_{0}^{2\pi}\int_{0}^{\pi} \frac{\left|\frac{\partial A_{\hat{\theta}(n,\varphi,\theta)}}{\partial \varphi}-\cos\theta A_{\hat{\varphi}(n,\varphi,\theta)} \right|^2}{\sin\theta}\rmd\theta\rmd\varphi \notag\\-
			=&\int_{0}^{2\pi}\int_{0}^{\pi}  (C_{n}^{n})^2\left\{ (2n-1)^2n^2(2n-3)P_{n-2}^{n-2}(\cos\theta)P_{n-1}^{n-1}(\cos\theta)+2n(2n-1)(2n+3)P_{n}^{n-1}(\cos\theta)\right.\notag\\
			&\times P_{n-2}^{n-2}(\cos\theta)+\left.(2n-1)P_{n-1}^{n-2}(\cos\theta)P_{n}^{n-1}(\cos\theta) \right\} \rmd\theta\rmd\varphi\notag\\
			=&n^4+\mathcal{O}\left(n^3\right).\label{eq:partial Asin}
		\end{align}
		Furthermore, by utilizing the similar argument for \eqref{eq:partial Asin}, it yields that
		\begin{equation}\label{eq:partial Avarphi sin}
			\int_{0}^{2\pi}\int_{0}^{\pi} 	\frac{\left|\frac{\partial A_{\hat{\varphi}(n,\varphi,\theta)}}{\partial \varphi}+\cos \theta A_{\hat{\theta}(n,\varphi,\theta)}\right|}{\sin\theta}\rmd\theta\rmd\varphi
			=\frac{n^3}{2}+\mathcal{O}\left(n^2\right). 
		\end{equation}
		Substituting \eqref{eq:P theta, varepsilon defination}, \eqref{eq:Atheta} and \eqref{eq:Avarphi} into \eqref{eq:C1n=}, we can get the estimate for $C_1(n)$  in \eqref{eq:C1nC2nC3n}.  Combining with \eqref{eq:Q def}, \eqref{eq:C2n=}, \eqref{eq:partial Atheta}, \eqref{eq:partial Avarphi}, \eqref{eq:partial Asin} and \eqref{eq:partial Avarphi sin}, we can derive the estimate for $C_2(n)$  in \eqref{eq:C1nC2nC3n}. 
		In  view of $C_{n,\delta,\tau}$ given by \eqref{eq:C_nm delta definition}, using the asymptotic estimations of $C_i(n)$ ($i=1,2$) given by \eqref{eq:C1nC2nC3n}, we can derive the corresponding  asymptotic analysis for $\|\nabla \bmf u\|_{L^2(\mathcal{M}_{-}^{1-\gamma_1}(\partial D))^3}^2$  with respect to $\omega$ as follows 
		\begin{align}\label{eq:nabla u in sigma1 new}
			\|\nabla \bmf u\|_{L^2(\mathcal{M}_{-}^{1-\gamma_1}(\partial D))^3}^2
			=&\frac{(2n+1) n^2 (2n^2-2n+1)\left|f_{n,n}\right|^2\delta^2\rho^n (1-\gamma_1^{2n+1}) \omega^{2n}}{[(2n+1)!!]^2 [\delta(n+2)+n-1]^2 \mu^n} \left(1+\mathcal{O}
			\left( \omega^2\right)\right). 
		\end{align}
		In view of $\mathbf u^i$  given by \eqref{eq:incident wave in T_n^m}, by direct calculations, it follows directly that 
		\begin{align}\label{eq:u^i in Sigma_1}
			\| \bmf u^i\|_{L^2(D)^3}^2
			=\frac{4\pi n(n+1)\left|f_{n,n}\right|^2 \rho^n \omega^{2n}}{[(2n+1)!!]^2(2n+3)\mu^n}
			\left(1+\mathcal{O}\left(\omega^2\right)\right).
		\end{align}
		Combining  \eqref{eq:nabla u in sigma1 new} with \eqref{eq:u^i in Sigma_1}, for sufficiently large $n$, it holds that 
		\begin{align}\label{eq:uui431}
			&\frac{\|\nabla \bmf u\|_{L^2(\mathcal{M}_{-}^{1-\gamma_1}(\partial D))^3}^2}{\| \bmf u^i\|_{L^2(D)^3}^2} \notag\\
			= &\frac{\frac{(2n+1) n^2 (2n^2-2n+1)\delta^2\rho^n (1-\gamma_1^{2n+1}) \omega^{2n}}{[(2n+1)!!]^2 [\delta(n+2)+n-1]^2 \mu^n} \left(1+\mathcal{O}\left( \omega^2\right)\right)}{\frac{4\pi n(n+1) \rho^n \omega^{2n}}{[(2n+1)!!]^2(2n+3)\mu^n}\left(1+\mathcal{O}\left(\omega^2\right)\right)}\notag\\
			=&\frac{\frac{n(2n+1)(2n+3)(2n^2-2n+1)\delta^2 (1-\gamma_1^{2n+1})}{4\pi (n+1) [\delta(n+2)+n-1]^2} \left(1+\mathcal{O}\left( \omega^2\right)\right)}{1+\mathcal{O}\left(\omega^2\right)}  \notag\\
			= & \frac{n^2\delta^2 }{4\pi} \frac{(2+1/n)(2+3/n)(2-2/n+1/n^2)(1-\gamma_1^{2n+1})}{(1+1/n)(1+\delta -(1-2
				\delta)1/n)^2}(1+\mathcal O(\omega^2 ))\left(1-\mathcal{O}\left(\omega^2\right)\right)\notag \\
			\geq & \frac{ n^2\delta^2}{16 \pi }\left(1+ \mathcal O(\omega^2 )\right)
			\geq \frac{ n^2\delta^2}{16 \pi }.  
		\end{align}

		\medskip
		
		\noindent{\bf Part 2.} In this part, we shall prove $\frac{\|\nabla \bmf  u^s\|_{L^2(\mathcal{M}_{+}^{\gamma_2-1}(\partial D))^3}}{\| \bmf u^i\|_{L^2(D)^3}}\geqslant\frac{n^2}{9\sqrt{\pi}}$. We note that by combining the scattering problem \eqref{eq:system} with the integral representation of the solution in \eqref{eq:total field}. We choose the suitable incident wave defined in \eqref{eq:u^i 4.1}, the scattering problem can be reformulated as solving a linear system represented by \eqref{eq:BX=b}. In view of \eqref{eq:jn expansion}, \eqref{eq:hn expansion},  \eqref{eq:S partial outside ball} and \eqref{eq:x_nm}, we can derive the scatter field. 
		\begin{align}
			\bmf u^s=\mathcal{S}_{\partial D}^\omega\left[\boldsymbol{\varphi_{2}}\right](\bmf x)
			=\left[G_{n,\rho,\mu,\delta}\omega^{2n+1}+\mathcal{O}\left(G_{n,\rho,\mu,\delta}^{\prime}\omega^{2n+3}\right)	\right] \sum_{m=-n}^n f_{n,n}  h_n\left(k_s  |\bmf x|\right)\mathcal{T}_n^n, \notag
		\end{align}
		where 
		\begin{align}
			G_{n,\rho,\mu,\delta}
			&=\frac{-\rmi (n-1)(1-\delta)\rho^{n+\frac{1}{2}}}{(2n+1)!!(2n-1)!![\delta(n+2)+n-1]\mu^{n+\frac{1}{2}}},\label{eq:G_nm delta definition}\\
			G_{n,\rho,\mu,\delta}^{\prime}
			&=\frac{\rmi \left[(n^2-1)(1-\delta)-\delta(2n+1)\right]\rho^{n+\frac{3}{2}}}{(2n+3)!!(2n+1)!!\mu^{n+\frac{3}{2}}}.\label{eq:Gnmprime}
		\end{align}
		After direct calculation, it is readily known that
		\begin{equation}\label{eq:nabla us in proof}
			\nabla \bmf u^s=
			\left[G_{n,\rho,\mu,\delta}\omega^{2n+1}+\mathcal{O}\left(G_{n,\rho,\mu,\delta}^{\prime}\omega^{2n+3}\right)\right]\sum_{m=-n}^n f_{n,n}\nabla\left(h_n({k}_s r) \mathcal{T}_n^n(\theta, \varphi)\right),
		\end{equation}
		where $G_{n,\rho,\mu,\delta}$ and $G_{n,\rho,\mu,\delta}^{\prime}$ are defined in \eqref{eq:G_nm delta definition} and \eqref{eq:Gnmprime}, respectively. Furthermore, in order to calculate $\nabla \bmf u^s$, we first introduce $\nabla\left(h_n(k_s r) \mathcal{T}_n^n(\theta, \varphi)\right)$. By combining \eqref{eq:hn qiudao}, \eqref{eq:Tnm expansion in theta and varphi} and \eqref{eq:grad Tnm}, it holds that 
		\begin{align}\label{eq:nabla hnTnm}
			&\nabla \left(h_n(k_s r) \mathcal{T}_n^n(\theta, \varphi)\right)
				= A_{\hat{\theta}(n,\varphi,\theta)}
				\left(h_{n-1}(k_s r) {k}_s-\frac{n+2}{r} h_n(k_s r)\right) \hat{r} \otimes \hat{\theta}
				+A_{\hat{\varphi}(n,\varphi,\theta)}\left(h_{n-1}({k}_s r) {k}_s \right. \notag\\
				&\left.-\frac{n+2}{r} h_n({k}_s r)\right) 
				\left(h_{n-1}({k}_s r) {k}_s-\frac{n+2}{r} h_n({k}_s r)\right) \hat{r} \otimes \hat{\varphi}
				+\left[\frac{1}{r} \frac{\partial A_{\hat{\theta}(n,\varphi,\theta)}}{\partial \theta} \hat{\theta} \otimes 
				\hat{\theta}\right.+\left(\frac{1}{r \sin \theta} \right. \notag\\
				&\left.\times\frac{\partial A_{\hat{\theta}(n,\varphi,\theta)}}{\partial \varphi}-\cot \theta\frac{A_{\hat{\varphi}(n,\varphi,\theta)}}{r}\right) \hat{\theta} \otimes \hat{\varphi}
				+\frac{1}{r} \frac{\partial A_{\hat{\varphi}(n,\varphi,\theta)}}{\partial \theta} \hat{\varphi} \otimes \hat{\theta}+\left(\frac{1}{r \sin \theta} \frac{\partial A_{\hat{\varphi}(n,\varphi,\theta)}}{\partial \varphi}\right.\notag\\
				&+\left.\left.\cot \theta \frac{A_{\hat{\theta}(n,\varphi,\theta)}}{r}\right)
				\hat{\varphi} \otimes \hat{\varphi}\right]\cdot h_n(k_s r).
			\end{align}
			Here, $A_{\hat{\theta}(n,\varphi,\theta)}$ and $A_{\hat{\varphi}(n,\varphi,\theta)}$ are defined in \eqref{eq:Tnm expansion in theta and varphi}, while $\hat{r}$, $\hat{\theta}$, and $\hat{\varphi}$ are defined correspondingly in  \eqref{eq:r hat theta hat varphi hat definition}. Combining with \eqref{eq:ks kp defination}, \eqref{eq:hn expansion}, \eqref{eq:P theta, varepsilon defination}-\eqref{eq:C1nC2nC3n}, \eqref{eq:G_nm delta definition},  \eqref{eq:nabla us in proof} and \eqref{eq:nabla hnTnm}, when $n$ is sufficiently large, we can derive that 
			%
			\begin{align}
				&\|\nabla \bmf  u^s\|_{L^2(\mathcal{M}_{+}^{\gamma_2-1}(\partial D))^3}^2
				=\int_{\mathcal{M}_{+}^{\gamma_2-1}(\partial D)} \left[G_{n,\rho,\mu,\delta}^2\omega^{4n+2}+\mathcal{O}\left(G_{n,\rho,\mu,\delta}G_{n,\rho,\mu,\delta}^{\prime}\omega^{4n+3}\right)\right]
				\left\{ P(\theta,\varphi)\left[h_{n-1}({k}_s r)  \right.\right.\notag\\
				&\left.\left. \times {k}_s r-(n+2) h_n({k}_s r)\right]^2+h_n^2(k_s r)Q(\theta,\varphi)\right\}\rmd r \rmd \theta\rmd \varphi\notag\\
				&=\frac{2n^2(n^2+2n+2)(n-1)^2 \left|f_{n,n}\right|^2 (1-\delta)^2\rho^{n}\omega^{2n}(\gamma_{2}^{2n+1}-1)}
				{[(2n+1)!!]^2(2n+1)[\delta(n+2)+n-1]^2\mu^{n}\gamma_2^{2n+1}}
				\left(1+\mathcal{O}\left(\omega^2\right)\right). \label{eq:nabla u^s in Sigma_2}
			\end{align}
			%
			Based on \eqref{eq:nabla u^s in Sigma_2} and \eqref{eq:u^i in Sigma_1}, it is readily to know that 
			\begin{align} \label{eq:usui437}
				\frac{\|\nabla \bmf  u^s\|_{L^2(\mathcal{M}_{+}^{\gamma_2-1}(\partial D))^3}^2}{\| \bmf  u^i\|_{L^2(D)^3}^2}
				=&\frac{\frac{n(2n+3)(n^2+2n+2)(n-1)^2(1-\delta)^2(\gamma_{2}^{2n+1}-1)}{2 \pi (2n+1)(n+1)[\delta(n+2)+n-1]^2\gamma_{2}^{2n+1}}\left(1+\mathcal{O}\left(\omega^2\right)\right)}{\left(1+\mathcal{O}\left(\omega^2\right)\right)} \notag\\
				=& \frac{n^2}{2 \pi}\frac{(2+3/n)(1+2/n+2/n^2)(1-1/n)^2}{(2+1/n)(1+1/n)[1+\delta-(1/n-2\delta/n)]^2} \left(1+\mathcal{O}\left(\omega^2\right)\right)\left(1-\mathcal{O}\left(\omega^2\right)\right)  \notag\\
				\geqslant&\frac{ n^2}{81 \pi} \left(1+\mathcal{O}\left(\omega^2\right)\right)
				\geqslant\frac{ n^2}{81 \pi} .
			\end{align}

			The proof is complete.
		\end{proof}

Theorem~\ref{thm:nabla u in thm} shows that for the incident wave in \eqref{eq:u^i 4.1} with index $n$ satisfying \eqref{eq:n4.2}, boundary localization and surface resonance occur simultaneously. This confirms the feasibility of quasi-Minnaert resonance. On the other hand, Proposition~\ref{pro:4.1} rigorously proves that surface resonance can arise without boundary localization.

	\begin{prop}\label{pro:4.1}
	 
	 Under the same assumptions as Theorem \ref{thm:nabla u in thm}, recall that \(\delta\) represents the  bulk modulus or shear modulus ratio between the hard elastic inclusion \(D\) and the soft elastic background \(\mathbb{R}^3 \setminus D\), if the parameter \( n \), defining the incident wave \(\mathbf{u}^i\) given by \eqref{eq:u^i 4.1}, only satisfies
   \begin{align}\label{eq:n delta con}
	n \geqslant \frac{1}{\delta^2},
  \end{align}
   then the corresponding total field $\mathbf{u}$ in $D$ and the scattered field $\mathbf{u}^s$ in $\mathbb{R}^3 \setminus D$ satisfy \eqref{eq:nablau nabla us gg 1}. Namely, the surface resonance occurs.
    
\end{prop}

   \begin{proof}
   	Following the approach of Theorem~\ref{thm:nabla u in thm}, we demonstrate that the total field $\mathbf{u}$ in $D$ and the scattered field $\mathbf{u}^s$ in $\mathbb{R}^3 \setminus D$ satisfy \eqref{eq:uui431} and \eqref{eq:usui437}, respectively. 
   	Combining \eqref{eq:uui431} and \eqref{eq:n delta con}, and noting that $\delta \ll 1$ as defined in \eqref{eq:lambda mu rho O1}, it can be readily shown that
   	\begin{align}\label{eq:35}
   		\frac{\|\nabla \bmf{u}\|_{L^2(\mathcal{M}_{-}^{1-\gamma_1}(\partial D))^3}}{\| \bmf{u}^i\|_{L^2(D)^3}} \geqslant \left(\frac{1}{4\sqrt{\pi}\delta}\right) \gg 1.
   	\end{align}
   	Similarly, the second inequality in \eqref{eq:nablau nabla us gg 1} holds via an analogous argument.
   \end{proof}


  When the material parameter $\delta$ (defining the Lam\'{e} parameter contrast) is held constant, Theorem~\ref{thm:nabla u in thm} demonstrates that selecting an appropriate index $n$ for the incident wave $\mathbf{u}^i$ in \eqref{eq:u^i 4.1} enables the concurrent induction of boundary localization and surface resonance for the total field $\mathbf{u}|_D$ and scattered field $\mathbf{u}^s|_{\mathbb{R}^3\setminus D}$. Similar to Corollary~\ref{cor:3.2}, Proposition~\ref{pro:4.2} establishes that tuning the material parameter $\delta$ similarly achieves both effects simultaneously.

\begin{prop}\label{pro:4.2}  
	Under the assumptions of Theorem~\ref{thm:nabla u in thm} with fixed $\gamma_1$ and $\gamma_2$, if the incident wave index $n$ in \eqref{eq:u^i 4.1} and the high contrast parameter $\delta$ satisfy \eqref{eq:n delta con} and \eqref{eq:327}, respectively, then:  
	\begin{itemize}  
		\item the internal total field $\mathbf{u}|_D$ and external scattered field $\mathbf{u}^s|_{\mathbb{R}^3\setminus D}$ exhibit boundary localization as defined in \eqref{eq:thm 3.1}; 
		\item the internal total field $\mathbf{u}|_D$ and external scattered field $\mathbf{u}^s|_{\mathbb{R}^3\setminus D}$ exhibit surface resonance as characterized by \eqref{eq:nablau nabla us gg 1}.  
	\end{itemize}
The simultaneous manifestation of these phenomena is referred to as \textit{quasi-Minnaert resonance} as per Definition~\ref{def:quasi minnaert}.
\end{prop}

   \begin{proof}
   		When condition \eqref{eq:327} is satisfied, \eqref{eq:326} follows immediately. Furthermore, the combination of \eqref{eq:n delta con} with the condition $\delta \ll 1$ ensures that \eqref{eq:n1n2 max} holds. Theorem~\ref{thm:internal surface localization for scattering problem} implies that both the total field $\mathbf{u}$ in $D$ and the scattered field $\mathbf{u}^s$ in $\mathbb{R}^3 \setminus D$ exhibit boundary localization. Moreover, since the incident field index $n$ satisfies \eqref{eq:n delta con}, Proposition~\ref{pro:4.1} immediately yields the surface resonance phenomenon. Consequently, the internal total field and external scattered field collectively manifest the quasi-Minnaert resonance.
   \end{proof}

Finally, we conclude that the quasi-Minnaert resonance depends critically on both the incident field $\mathbf{u}^i$ specified in \eqref{eq:u^i 4.1} and the material contrast parameter $\delta$ defined in \eqref{eq:lambda mu rho}. The following theorem demonstrates that for a fixed material configuration, where $\delta$ represents the ratio of the Lam\'e parameters between the hard inclusion and soft background medium, and for sufficiently small $\varepsilon$ (characterizing the boundary localization level), there exists an appropriate incident wave index $n$ that induces stress concentration.

	\begin{thm}\label{thm:Eu definition in thm}
		Consider the elastic scattering problem \eqref{eq:system} and let $(D;\tilde{\lambda},\tilde{\mu},\tilde{\rho})$ be the unit ball embedded in a homogeneous elastic medium $(\mathbb R^3 \setminus \overline D;\lambda,\mu,\rho)$ in $ \mathbb{R}^3 $. We introduce $E(\bmf u)$ and $E(\mathbf{u}^s)$ as
		\begin{align}
			E(\bmf u)=\int_{\mathcal{M}_{-}^{1-\gamma_1}(\partial D)} \sigma(\bmf u): \nabla \overline{\bmf u} \rmd \mathrm{x}, 
			\quad 
			E(\bmf u^s)=\int_{\mathcal{M}_{+}^{\gamma_2-1}(\partial D)} \sigma(\bmf u^s): \nabla \overline{\bmf u^s} \rmd \mathrm{x},\notag
		\end{align}
		where  
		\begin{equation}\label{eq:sigma u definition}
			\sigma(\bmf u)=\tilde{\lambda}(\nabla \cdot \bmf u) \mathcal{I}+\tilde{\mu}\left(\nabla \bmf u+\nabla \bmf u^{\top}\right), \quad
			\sigma(\bmf u^s)={\lambda}(\nabla \cdot \bmf u^s) \mathcal{I}+{\mu}\left(\nabla \bmf u^s+(\nabla \bmf u^s)^{\top}\right).\notag
		\end{equation}
		Here, $\tilde{\lambda}$, $\tilde{\mu}$, $\lambda$ and $\mu$ are defined in \eqref{eq:lambda mu rho}, and $\mathcal{I}$ represents the identity matrix. 
		Recall that  $\mathcal{M}_{-}^{1-\gamma_1}(\partial D)$ and $\mathcal{M}_{+}^{\gamma_2-1}(\partial D)$, are introduced in~\eqref{eq:Mdef}  with $\gamma_1\in(0,1)$ and $\gamma_2\in(1,R)$, respectively. The operation ``:'' represents the inner product of two matrices. Under the assumptions \eqref{eq:lambda mu rho}-\eqref{eq:assp omega}, for a sufficiently small fixed parameter $\varepsilon \ll 1$ representing the level of boundary localization, if the index $n$ of the incident wave specified in \eqref{eq:u^i 4.1} satisfies  
		\begin{align}\label{eq:n46}
			n \geqslant \max\{ n_1,n_2, 1/\delta\},
		\end{align}
		where $\delta$ denotes the ratio of the Lam\'e parameters of the hard elastic inclusion $D$ to those of the soft elastic background $\mathbb{R}^3 \setminus D$ as given in \eqref{eq:lambda mu rho} and $n_1$, $n_2$ are defined in \eqref{eq:n1 n2 def}, then the corresponding total field $\mathbf{u}|_D$ and scattered field $\mathbf{u}^s|_{\mathbb{R}^3 \setminus \overline{D}}$ demonstrate stress concentration. Namely, it holds that
		\begin{align}
			\frac{E(\bmf u)}{\| \bmf u^i\|_{L^2(D)^3}^2}
			&\geqslant\frac{4 n^2 \delta}{27 \pi}\gg 1, 
			\qquad
			\frac{E(\mathbf{u}^s)}{\| \bmf u^i\|_{L^2(D)^3}^2}
			\geqslant\frac{n^2}{81\pi}\gg 1. \label{eq:Eu u^i ratio}
		\end{align}
	\end{thm}

\begin{rem}

Under the condition of a sufficiently small boundary localization level \(\varepsilon \ll 1\), the definitions of \(n_1\) and \(n_2\) in \eqref{eq:n1 n2 def} ensure that \(n\) complies with \eqref{eq:n44}. For the boundary layer parameters \(\gamma_1 \in (0,1)\) and \(\gamma_2 \in (1,R)\), if the deviations \(|\gamma_j - 1|\) (\(j=1,2\)) are sufficiently small relative to \(\epsilon\), the constants \(n_j\) become arbitrarily large (\(n_j \gg 1\)) even for fixed \(\varepsilon\).   From \eqref{eq:n46}, we derive the scaling relation \({n}\delta = \mathcal{O}(1)\), which further implies \(n^2\delta = \mathcal{O}(n) \gg 1\). This directly validates the first inequality in \eqref{eq:Eu u^i ratio}, while the second follows from \eqref{eq:n44}.   Consequently, the total field \(\mathbf{u}|_D\) and scattered field \(\mathbf{u}^s|_{\mathbb{R}^3 \setminus \overline{D}}\) exhibit  stress concentration. This analysis rigorously establishes that the phenomenon arises from the interplay between the incident wave field \(\mathbf{u}^i\) (given in \eqref{eq:u^i 4.1}) and the material contrast parameter \(\delta\) defined in \eqref{eq:lambda mu rho}.

\end{rem}

In the following we give the proof of Theorem \ref{thm:Eu definition in thm}. 

\begin{proof}
The proof shall be divided  into two parts.
	
	\medskip
	
	\noindent{\bf Part 1.} This part is dedicated to the proof of the first equation of  \eqref{eq:Eu u^i ratio}. 
	Since the incident wave $\mathbf u^i$ is given by \eqref{eq:u^i 4.1}, by Lemma \ref{lem:total field and scatter field} we have the formula \eqref{eq:u single potential inside the unit sphere} for $\mathbf u|_D$ with respect to $\omega$. Using the fact that $\nabla \cdot \mathcal T^n_n=0$, it is readily to know that
	\begin{align}\label{eq:Eu 435}
		E(\bmf u)
		=\tilde \mu 	\|\nabla \bmf u\|_{L^2(\mathcal{M}_{-}^{1-\gamma_1}(\partial D))^3}^2+\int_{\mathcal{M}_{-}^{1-\gamma_1}(\partial D)}\tilde{\mu} \cdot {\rm tr}(\nabla \bmf u \nabla \mathbf {\bar u})~\rm d \mathbf x. 
	\end{align}
	Hence, from \eqref{eq:nabla jn Tnm} and \eqref{eq:nabla u definition}, we can deduce that 
	%
	\begin{align}\label{eq:nabla bmf u cdot nabla bmf u}
		&\nabla \bmf u  \nabla \overline{\bmf  u}\\
		=&\left\{\left(j_{n-1}(\tilde{k}_s r) \tilde{k}_s-\frac{n+2}{r} j_n(\tilde{k}_s r)\right)^2\left[\left(f_{n,n}A_{\hat{\theta}(n,\varphi,\theta)}\right)\cdot\left(f_{n,n}\frac{1}{r} \frac{\partial A_{\hat{\theta}(n,\varphi,\theta)}}{\partial \theta}\right)\right.\right.\notag\\
		&+\left.\left(f_{n,n}A_{\hat{\varphi}(n,\varphi,\theta)}\right)\cdot\left(f_{n,n}\frac{1}{r} \frac{\partial A_{\hat{\varphi}(n,\varphi,\theta)}}{\partial \theta}\right)\right] \hat{r} \otimes \hat{\theta}
		+\left(j_{n-1}(\tilde{k}_s r) \tilde{k}_s-\frac{n+2}{r} j_n(\tilde{k}_s r)\right)^2\notag\\
		&\times \left[\left(f_{n,n}A_{\hat{\theta}(n,\varphi,\theta)}\right)\cdot\left(f_{n,n}\frac{1}{r} \frac{\partial A_{\hat{\theta}(n,\varphi,\theta)}}{\partial \varphi}\right)\right.+\left(f_{n,n}A_{\hat{\varphi}(n,\varphi,\theta)}\right)\notag\\
		&\times \left.\left(f_{n,n}\frac{1}{r} \frac{\partial A_{\hat{\varphi}(n,\varphi,\theta)}}{\partial \varphi}\right)\right]\hat{r} \otimes \hat{\varphi}
		+\left\{\left[f_{n,n}\left(\frac{1}{r\sin\theta}\frac{\partial A_{\theta}}{\partial\varphi}-\cot\theta\frac{A_{\hat{\varphi}(n,\varphi,\theta)}}{r}\right)\right]\cdot\right.\notag\\
		&\left[ f_{n,n}\left(\frac{1}{r}\frac{\partial A_{\theta}}{\partial\theta}+\frac{1}{r\sin\theta}\frac{\partial A\varphi}{\partial\varphi}+\cot\theta\frac{A_{\hat{\theta}(n,\varphi,\theta)}}{r}\right) \right]\hat{\theta} \otimes \hat{\varphi}
		+\left(f_{n,n}\frac{1}{r}\frac{\partial A_{\hat{\varphi}(n,\varphi,\theta)}}{\partial\theta}\right)\notag\\
		&\times \left[f_{n,n}\left(\frac{1}{r\sin\theta}\frac{\partial A_{\hat{\theta}(n,\varphi,\theta)}}{\partial\varphi}-\cot\theta\frac{A_{\hat{\varphi}(n,\varphi,\theta)}}{r}\right)\right]\hat{\theta} \otimes \hat{\theta}
		+\left(f_{n,n}\frac{1}{r}\frac{\partial A_{\hat{\varphi}(n,\varphi,\theta)}}{\partial\theta}\right)\notag\\
		&\times \left[f_{n,n}\right.\left(\frac{1}{r}\frac{\partial A_{\hat{\theta}(n,\varphi,\theta)}}{\partial\theta}+\frac{1}{r\sin\theta}\frac{\partial A_{\hat{\varphi}(n,\varphi,\theta)}}{\partial\varphi}\right.
		\left.\left.+\cot\theta\frac{A_{\hat{\theta}(n,\varphi,\theta)}}{r}\right)\right] \hat{\varphi} \otimes \hat{\theta}\notag\\
		&+\left\{\left(f_{n,n}\frac{1}{r}\frac{\partial A_{\hat{\varphi}(n,\varphi,\theta)}}{\partial\theta}\right)\cdot\left[f_{n,n}\left(\frac{1}{r\sin\theta}\frac{\partial A_{\hat{\theta}(n,\varphi,\theta)}}{\partial\varphi}-\cot\theta\frac{A_{\hat{\varphi}(n,\varphi,\theta)}}{r}\right)\right]\right.\notag\\
		&\left.\left.\left.+\left[f_{n,n}\left(\frac{1}{r\sin\theta}\frac{\partial A_{\hat{\varphi}(n,\varphi,\theta)}}{\partial\varphi}+\cot\theta\frac{A_{\hat{\varphi}(n,\varphi,\theta)}}{r}\right)\right]^{2}\right\}\hat{\varphi} \otimes \hat{\varphi}\right\}\cdot j_n(\tilde{k}_s r)
		\right\}\cdot\notag\\
		&\left[C_{n,\delta,\tau}+\mathcal{O}\left(  C_{n,\delta,\tau}^{\prime} \omega^2\right)\right]^2,\notag
	\end{align}
	where $C_{n,\delta,\tau}$ and $C_{n,\delta,\tau}^{\prime}$ are defined in \eqref{eq:C_nm delta definition}, $A_{\hat{\varphi}(n,\varphi,\theta)}$ and $A_{\hat{\theta}(n,\varphi,\theta)}$ are defined in \eqref{eq:A_theta A_varphi}, and $f_{n,n}\in \mathbb C$ is a constant defined in \eqref{eq:u^i 4.1}.  
	After direct calculations, it can be derived that 
	\begin{align}
		&\mathrm{tr}(\nabla \bmf u  \nabla\overline{\bmf u})
		=\frac{\left[C_{n,\delta,\tau}+\mathcal{O}\left(  C_{n,\delta,\tau}^{\prime} \omega^2\right)\right]\cdot j_n(\tilde{k}_s r)}{r^{2}}
		\left\{2\left| f_{n,n}\left(- C_{n}^{n}e^{\rmi n\varphi}\frac{\partial^{2}P_{n}^{n}(\cos\theta)}{\partial^{2}\theta}\right)\right|\right. \notag\\
		&\times \left| f_{n,n}\left(-\frac{n^{2}}{\sin^{2}\theta}Y_{n}^{n}(\theta,\varphi)+\cot\theta C_{n}^{n}e^{\rmi n\varphi}\frac{\partial P_{n}^{n}(\cos\theta)}{\partial\theta}\right)\right|+\left.\left| f_{n,n}A_{\hat{\varphi}(n,\varphi,\theta)}\frac{\rmi n+\cos \theta}{\sin \theta}\right|^{2}\right\}.\label{eq:tr nabla u nabla u}
	\end{align}
	%
	Recall that $C_{n,\delta,\tau}$ and $C_{n,\delta,\tau}^{\prime}$ are defined in \eqref{eq:C_nm delta definition}. Substituting \eqref{eq:nabla u expansion} into \eqref{eq:Eu 435}, due to \eqref{eq:tr nabla u nabla u}, it yields that

	\begin{align}\label{eq:Eu in Sigma1}
		E(\bmf u)
		&=\frac{\mu }{\delta} \int_{\mathcal{M}_{-}^{1-\gamma_1}(\partial D)} \left[\left(j_{n-1}(\tilde{k}_s r) \tilde{k}_s r-(n+2)j_n(\tilde{k}_s r)\right)^2 P(n,\theta,\varphi)+ j_n^2(\tilde{k}_s r) M(n,\theta,\varphi)\right]\notag\\
		&\quad\times \left[C_{n,\delta,\tau}^2+\mathcal{O}\left(  C_{n,\delta,\tau}C_{n,\delta,\tau}^{\prime} \omega^2\right)\right]\rmd  r \rmd \theta \rmd\varphi.
	\end{align}
	where 
	\begin{align}
		M(n,\theta,\varphi)
		=&Q(n,\theta,\varphi)+
		\left\{2\left| f_{n,n}\left(- C_{n}^{n}e^{\rmi n\varphi}\frac{\partial^{2}P_{n}^{n}(\cos\theta)}{\partial^{2}\theta}\right)\right|\right.\notag\\
		&\times \left| f_{n,n}\left(-\frac{n^{2}}{\sin^{2}\theta}Y_{n}^{n}(\theta,\varphi)+\cot\theta C_{n}^{n}e^{\rmi n\varphi}\frac{\partial P_{n}^{n}(\cos\theta)}{\partial\theta}\right)\right|\notag\\
		&+\left.\left| f_{n,n}A_{\hat{\varphi}(n,\varphi,\theta)}\frac{\rmi n+\cos \theta}{\sin \theta}\right|^{2}\right\}\cdot\sin\theta,\label{eq:M(theta,varphi) defination}\\
		C_3(n)=&\int_0^{2\pi}\int_0^{\pi}M(\theta,\varphi)\rmd \theta\rmd \varphi, \notag 
	\end{align}
	Here $A_{\hat{\theta}(n,\varphi,\theta)}$ and $A_{\hat{\varphi}(n,\varphi,\theta)}$ are defined in \eqref{eq:A_theta A_varphi}, $f_{n, n}\in \mathbb C$ is a constant defined in \eqref{eq:u^i 4.1}, and $Y_n^n(\theta,\varphi)$ and $C_n^n$ are defined in \eqref{eq:ynm and cnm def}. 
	By utilizing the similar argument for deriving the asymptotic estimations of $C_1(n)$ and $C_2(n)$ given by \eqref{eq:C1nC2nC3n},  we can rigorously demonstrate that 
	\begin{align}\label{eq:C_3n def}
		C_3(n)=n^4 \left| f_{n,n}\right|^2+\mathcal{O}\left(n^3 \left| f_{n,n}\right|^2\right).
	\end{align}
	Substituting \eqref{eq:ks kp defination}, \eqref{eq:tilde ks relationship with ks}, \eqref{eq:jn expansion} and \eqref{eq:C_3n def} into \eqref{eq:Eu in Sigma1}, we can derive that
	\begin{align}
		E(\bmf u)
		=&\frac{C_{n,\delta,\tau}^2\mu C_1(n)}{\delta}\int_{\gamma_{1}}^{1}\frac{(n-1)^2(\tilde{k}_{s}r)^{2n}}{[(2n+1)!!]^2}\rmd r 
		+\frac{C_{n,\delta,\tau}^2 \mu C_3(n)}{\delta} \int_{\gamma_{1}}^{1}\frac{(\tilde{k}_s r)^{2n}}{[(2n+1)!!]^2}\rmd r \notag\\
		&+\mathcal{O}\left(\frac{C_{n,\delta,\tau}C_{n,\delta,\tau}^{\prime}C_3(n)\rho^n\omega^{2n+2}}{[(2n+1)!!]^2(2n+1)\delta\mu^{n-1}}\right)\notag\\
		=&\frac{n^2(2n^2-2n+1)(2n+1)\left|f_{n,n}\right|^2(1-\gamma_1^{2n+1})\delta\rho^n\omega^{2n}}{[(2n+1)!!]^2[\delta(n+2)+n-1]^2\mu^{n-1}}\left(1+\mathcal{O}\left(\omega^2\right)\right). \label{eq:Eu result}
	\end{align}
	 
	%
	Combining with \eqref{eq:lambda mu rho O1}, \eqref{eq:C1nC2nC3n} \eqref{eq:u^i in Sigma_1} and \eqref{eq:Eu result}, when $n$ is sufficiently large, it yields that
	\begin{align}\label{eq:445}
		\frac{E(\bmf u)}{\| \bmf u^i\|_{L^2(D)^3}^2}
		&=\frac{\frac{n (2n+1) (2n+3) (2n^2-2n+1) (1-\gamma_1^{2n+1})\delta}{4 \pi [\delta(n+2)+n-1]^2 (n+1) } \left(1+\mathcal{O}\left( \omega^2\right)\right)}
		{1+\mathcal{O}\left(\omega^2\right)}\notag\\
		&= \frac{2 n^2 \delta}{ \pi} \frac{(1+1/2n)(1+3/2n)(1-1/n+1/2n^2)}{(1+1/n)[1+\delta-(1/n-2\delta/n)]^2} \left(1+\mathcal{O}\left( \omega^2\right)\right) 	\left(1-\mathcal{O}\left(\omega^2\right)\right)   \notag\\
	 	&\geqslant\frac{4n^2 \delta}{27 \pi} \left(1+\mathcal{O}\left( \omega^2\right)\right) \geqslant\frac{4n^2 \delta}{27 \pi} . 
	\end{align}

	\medskip
	\noindent{\bf Part 2.} This part is dedicated to the proof of the second formula of \eqref{eq:Eu u^i ratio}. Since the incident wave $\mathbf u^i$ is given by \eqref{eq:u^i 4.1}, by Lemma \ref{lem:total field and scatter field} we have the formula \eqref{eq:u single potential inside the unit sphere} for $\mathbf u^s|_{\mathbb R^3 \setminus \overline D}$ with respect to $\omega$. Using the fact that $\nabla \cdot \mathcal T^n_n=0$, it is readily to know that
	\begin{align}\label{eq:Eu^s443}
		E(\bmf u^s)
		= \mu  \|\nabla \bmf u^s\|_{L^2(\mathcal{M}_{+}^{\gamma_2-1}(\partial D))^3}^2+\int_{\mathcal{M}_{+}^{\gamma_2-1}(\partial D)} \mu \cdot {\rm tr}(\nabla \bmf u^s \nabla \mathbf {\overline {u^s}})~\rm d \mathbf x. 
	\end{align}
	Hence, by a similar argument for \eqref{eq:nabla bmf u cdot nabla bmf u}, due to \eqref{eq:G_nm delta definition} and \eqref{eq:nabla us in proof}, one can obtain that 
	\begin{align}
		&\nabla \bmf u^s \nabla \overline{\bmf  u^s}\notag\\
		=&\left\{\left(h_{n-1}({k}_s r) {k}_s-\frac{n+2}{r} h_n({k}_s r)\right)^2\left[\left(f_{n,n}A_{\hat{\theta}(n,\varphi,\theta)}\right)\cdot\left(f_{n,n}\frac{1}{r} \frac{\partial A_{\hat{\theta}(n,\varphi,\theta)}}{\partial \theta}\right)\right.\right.\notag\\
		&+\left(f_{n,n}A_{\hat{\varphi}(n,\varphi,\theta)}\right)\left.\left(f_{n,n}\frac{1}{r} \frac{\partial A_{\hat{\varphi}(n,\varphi,\theta)}}{\partial \theta}\right)\right] \hat{r} \otimes \hat{\theta}
		+\left(h_{n-1}(\tilde{k}_s r) \tilde{k}_s-\frac{n+2}{r} h_n(\tilde{k}_s r)\right)^2\notag\\
		&\times\left[\left(f_{n,n}A_{\hat{\theta}(n,\varphi,\theta)}\right)\left(f_{n,n}\frac{1}{r} \frac{\partial A_{\hat{\theta}(n,\varphi,\theta)}}{\partial \varphi}\right)\right.+\left(f_{n,n}A_{\hat{\varphi}(n,\varphi,\theta)}\right)\notag\\
		&\times\left.\left(f_{n,n}\frac{1}{r} \frac{\partial A_{\hat{\varphi}(n,\varphi,\theta)}}{\partial \varphi}\right)\right]\hat{r} \otimes \hat{\varphi}
		+\left\{\left[f_{n,n}\left(\frac{1}{r\sin\theta}\frac{\partial A_{\theta}}{\partial\varphi}-\cot\theta\frac{A_{\hat{\varphi}(n,\varphi,\theta)}}{r}\right)\right]\right.\notag\\
		&\times\left[ f_{n,n}\left(\frac{1}{r}\frac{\partial A_{\theta}}{\partial\theta}+\frac{1}{r\sin\theta}\frac{\partial A\varphi}{\partial\varphi}+\cot\theta\frac{A_{\hat{\theta}(n,\varphi,\theta)}}{r}\right) \right]\hat{\theta} \otimes \hat{\varphi}
		+\left(f_{n,n}\frac{1}{r}\frac{\partial A_{\hat{\varphi}(n,\varphi,\theta)}}{\partial\theta}\right)\notag\\
		&\times\left[f_{n,n}\left(\frac{1}{r\sin\theta}\frac{\partial A_{\hat{\theta}(n,\varphi,\theta)}}{\partial\varphi}-\cot\theta\frac{A_{\hat{\varphi}(n,\varphi,\theta)}}{r}\right)\right]\hat{\theta} \otimes \hat{\theta}
		+\left(f_{n,n}\frac{1}{r}\frac{\partial A_{\hat{\varphi}(n,\varphi,\theta)}}{\partial\theta}\right)\notag\\
		&\times\left[f_{n,n}\left(\frac{1}{r}\frac{\partial A_{\hat{\theta}(n,\varphi,\theta)}}{\partial\theta}+\frac{1}{r\sin\theta}\frac{\partial A_{\hat{\varphi}(n,\varphi,\theta)}}{\partial\varphi}+\cot\theta\frac{A_{\hat{\theta}(n,\varphi,\theta)}}{r}\right)\right] \hat{\varphi} \otimes \hat{\theta}\notag\\
		&+ \left\{\left(f_{n,n}\frac{1}{r}\frac{\partial A_{\hat{\varphi}(n,\varphi,\theta)}}{\partial\theta}\right)\left[f_{n,n}\left(\frac{1}{r\sin\theta}\frac{\partial A_{\hat{\theta}(n,\varphi,\theta)}}{\partial\varphi}-\cot\theta\frac{A_{\hat{\varphi}(n,\varphi,\theta)}}{r}\right)\right]\right.\notag\\
		&\left.\left.\left.+\left[f_{n,n}\left(\frac{1}{r\sin\theta}\frac{\partial A_{\hat{\varphi}(n,\varphi,\theta)}}{\partial\varphi}+\cot\theta\frac{A_{\hat{\varphi}(n,\varphi,\theta)}}{r}\right)\right]^{2}\right\}\hat{\varphi} \otimes \hat{\varphi}\right\}\cdot h_n^2(k_s r)
		\right\}\left[G_{n,\rho,\mu,\delta}\omega^{2n+1}\right.\notag\\
		&\left.+\mathcal{O}\left(G_{n,\rho,\mu,\delta}^{\prime}\omega^{2n+3}\right)\right]^2, \notag 
	\end{align}
	where $G_{n,\rho,\mu,\delta}$ and $G_{n,\rho,\mu,\delta}^{\prime}$ are defined in \eqref{eq:G_nm delta definition} and \eqref{eq:Gnmprime}, respectively. $f_{n,n}\in \mathbb C$ is a constant defined in \eqref{eq:u^i 4.1}. $A_{\hat{\theta}(n,\varphi,\theta)}$ and $A_{\hat{\varphi}(n,\varphi,\theta)}$ are defined in \eqref{eq:A_theta A_varphi}. After direct calculations, it can be derived that  
	%
	%
	\begin{align}
		&\mathrm{tr}(\nabla \bmf u^s \cdot \nabla\overline{\bmf  u^s})	\label{eq:tr nabla us}\\
		=&\left\{2\left|f_{n,n}\left(- C_{n}^{m}e^{\rmi m\varphi}\frac{\partial^{2}P_{n}^{n}(\cos\theta)}{\partial^{2}\theta}\right)\right|\right.\cdot\left| f_{n,n}\left(-\frac{m^{2}}{\sin^{2}\theta}Y_{n}^{m}(\theta,\varphi)+\cot\theta C_{n}^{m}e^{\rmi m\varphi}\right.\right.\notag\\
		&\times\left.\left.\left.\frac{\partial P_{n}^{n}(\cos\theta)}{\partial\theta}\right)\right|+\left| f_{n,n}A_{\hat{\varphi}(n,\varphi,\theta)}\frac{\rmi m+\cos \theta}{\sin \theta}\right|^{2}\right\}\cdot\frac{\left[G_{n,\rho,\mu,\delta}\omega^{2n+1}+\mathcal{O}\left(G_{n,\rho,\mu,\delta}^{\prime}\omega^{2n+3}\right)\right]^2 h_n(k_s r)}{r^{2}},\notag
	\end{align}
	where $Y_n^m(\theta,\varphi)$ and $C_n^m$ are defined in \eqref{eq:ynm and cnm def}. 
	Recall that $G_{n,\rho,\mu,\delta}$ and $G_{n,\rho,\mu,\delta}^{\prime}$ are defined in \eqref{eq:G_nm delta definition} and \eqref{eq:Gnmprime}, respectively. Substituting \eqref{eq:hn expansion}, \eqref{eq:nabla u^s in Sigma_2} and \eqref{eq:tr nabla us} into \eqref{eq:Eu^s443}, due to \eqref{eq:C1nC2nC3n} and \eqref{eq:C_3n def}, it yields that
	%
	\begin{align}\label{eq:Eu^s definition in simple}
		E(\bmf u^s) 
		=&\mu\left[G_{n,\rho,\mu,\delta}\omega^{2n+1}+\mathcal{O}\left(G_{n,\rho,\mu,\delta}^{\prime}\omega^{2n+3}\right)\right]^2 \int_{\mathcal{M}_{+}^{\gamma_2-1}(\partial D)} \left[ r^2\left(h_{n-1}({k}_s r) {k}_s-\frac{n+2}{r} h_n({k}_s r)\right)^2 \right.\notag\\
		&\left.\times P(n,\theta,\varphi)+h_n^2(k_s r)M(n,\theta,\varphi)\right]\rmd  r \rmd \theta \rmd\varphi  \\
		=&\frac{2n^2(n^2+2n+2)(n-1)^2\left|f_{n,n}\right|^2(1-\delta)^2\rho^{n}\omega^{2n}(\gamma_{2}^{2n+1}-1)}
		{[(2n+1)!!]^2(2n+1)[\delta(n+2)+n-1]^2\mu^{n-1}\gamma_2^{2n+1}}
		\left(1+\mathcal{O}\left(\omega^2\right)\right), \notag
	\end{align}
	where $P(n,\theta,\varphi)$ and $M(m,\theta,\varphi)$ are defined in \eqref{eq:P theta, varepsilon defination}, \eqref{eq:M(theta,varphi) defination}, respectably. 
	Therefore, according to \eqref{eq:u^i in Sigma_1} and \eqref{eq:Eu^s definition in simple}, after direct calculation, it yields that
	\begin{align}\label{eq:449}
		\frac{E(\mathbf{u}^s)}{\|\bmf u^i\|_{L^2(D)^3}^2}
		=&\frac{\frac{2n(n^2+2n+2)(n-1)^2(1-\delta)^2(2n+3)(\gamma_{2}^{2n+1}-1)\mu}{4\pi(2n+1)(n+1)[\delta(n+2)+n-1]^2\gamma_2^{2n+1}}
			\left(1+\mathcal{O}\left(\omega^2\right)\right)}
		{1+\mathcal{O}\left(\omega^2\right)} \notag\\
		=&\frac{n^2}{2\pi}\frac{(1+2/n+2/n^2)(1-1/n)^2(1+3/2n)}{(1+1/2n)(1+1/n)[1+\delta-(1/n-2\delta/n)]^2}
		\left(1+\mathcal{O}\left(\omega^2\right)\right)
		\left(1-\mathcal{O}\left(\omega^2\right)\right)\notag\\
		\geqslant&\frac{n^2}{81 \pi}\left(1+\mathcal{O}\left(\omega^2\right)\right)\geqslant\frac{n^2}{81 \pi}.
	\end{align}

	The proof is complete.
\end{proof}

Theorem~\ref{thm:Eu definition in thm} demonstrates that for the incident wave defined in \eqref{eq:u^i 4.1} with index $n$ satisfying \eqref{eq:n46}, boundary localization and stress concentration occur simultaneously. By contrast, Proposition \ref{prop:4.3} rigorously establishes that stress concentration can occur independently of boundary localization.

\begin{prop}\label{prop:4.3}
    	 
	 Under the same assumptions as in Theorem~\ref{thm:Eu definition in thm}, let $\delta$ represent the ratio of bulk moduli or shear moduli between the hard elastic inclusion $D$ and the homogenous elastic background medium $\mathbb{R}^3 \setminus D$. When the wave parameter $n$ characterizing the incident field $\mathbf{u}^i$ in \eqref{eq:u^i 4.1} satisfies condition \eqref{eq:n27}, the resulting total field $\mathbf{u}$ in $D$ and scattered field $\mathbf{u}^s$ in $\mathbb{R}^3 \setminus D$ exhibit the stress amplification relation \eqref{eq:Eu u^i ratio}, demonstrating the occurrence of stress concentration.

\end{prop}

\begin{proof}
	 Following the approach of Theorem~\ref{thm:Eu definition in thm}, we establish that the total field $\mathbf{u}$ within $D$ and the scattered field $\mathbf{u}^s$ in $\mathbb{R}^3 \setminus D$ satisfy \eqref{eq:445} and \eqref{eq:449}, respectively. Combining \eqref{eq:445} and \eqref{eq:n27}, and noting that $\delta \ll 1$ as defined in \eqref{eq:lambda mu rho}, we apply a similar argument to \eqref{eq:35} to readily show that the first inequality of \eqref{eq:Eu u^i ratio} holds. The same approach can be applied to prove that the second inequality of \eqref{eq:Eu u^i ratio} holds.
\end{proof}

Proposition~\ref{prop:4.3} reveals that condition \eqref{eq:n27} alone cannot rigorously guarantee the co-occurrence of boundary localization and stress concentration, as it does not directly entail \eqref{eq:thm 3.1}. Building upon the framework established in Corollary~\ref{cor:3.2}, Proposition~\ref{prop:4.4} demonstrates that properly tuning the material parameter $\delta$ can similarly demonstrate the simultaneous emergence of boundary localization and stress concentration phenomena.

\begin{prop}\label{prop:4.4}
	Under the same assumptions as Theorem~\ref{thm:nabla u in thm}, if the index $n$ of the incident field $\mathbf{u}^i$ specified in~\eqref{eq:u^i 4.1} satisfies~\eqref{eq:n27} and the high contrast material parameter $\delta$ satisfies \eqref{eq:327}, then the total field $\mathbf{u}$ in $D$ and the scattered field $\mathbf{u}^s$ in $\mathbb{R}^3 \setminus D$ exhibit both boundary localization and stress concentration simultaneously.
\end{prop}

\begin{proof}
	If the high contrast material parameter $\delta$ satisfies \eqref{eq:327}, it ensures the inequality \eqref{eq:326}, where $n_1$ and $n_2$ are defined in \eqref{eq:n1 n2 def}. By combining \eqref{eq:n27} with the index $n$ of the incident wave specified in~\eqref{eq:u^i 4.1}, which meets the condition \eqref{eq:n1n2 max}, we apply Theorem~\ref{thm:internal surface localization for scattering problem} to conclude that the total field $\mathbf{u}$ within $D$ and the scattered field $\mathbf{u}^s$ in $\mathbb{R}^3 \setminus D$ exhibit boundary localization. Moreover, given that the index $n$ of the incident field $\mathbf{u}^i$ satisfies~\eqref{eq:n27}, Proposition~\ref{prop:4.3} directly confirms the manifestation of stress concentration. Thus, the internal total field $\mathbf{u}$ in $D$ and the external scattered field $\mathbf{u}^s$ in $\mathbb{R}^3 \setminus D$ exhibit both boundary localization and stress concentration simultaneously.
\end{proof}

	Finally, in Remark~\ref{rem:4}, we rigorously establish that surface resonance necessarily induces stress concentration.

	\begin{rem}\label{rem:4}

		By combining the definitions of \(n_1\) and \(n_2\) in \eqref{eq:n1 n2 def} with the high-contrast parameter \(\delta \ll 1\) (representing the ratio of material parameters between the hard  inclusion $D$ and the soft background medium $\mathbb R^3 \setminus D$), we establish that the stress concentration condition \eqref{eq:n46} is less restrictive than the surface resonance condition \eqref{eq:n4.2}.  Remarkably, our analysis reveals that these distinct phenomena can be selectively achieved by appropriate choice of the incident wave index \(n\):
\begin{itemize}
    \item surface resonance occurs when \(n\) satisfies only \eqref{eq:n delta con};
    \item stress concentration manifests when \(n\) satisfies only \eqref{eq:n27}.
\end{itemize}
While neither condition guarantees boundary localization, the implication \eqref{eq:n delta con} $\Rightarrow$ \eqref{eq:n27} demonstrates that surface resonance invariably induces stress concentration when \(n\) fulfills \eqref{eq:n delta con}.
	\end{rem}

		\section*{Concluding Remarks}


	In this paper, we present a detailed theoretical analysis of quasi-Minnaert resonance and stress concentration for radial geometry inclusions. We demonstrate that the emergence of quasi-Minnaert resonance and stress concentration is intricately governed by two key factors: (i) the choice of the incident wave, particularly its index $n$,  and (ii) the material parameters, specifically the ratio of Lam\'e parameters between the hard inclusion and the soft background. Our findings contribute significantly to the theoretical understanding of composite materials and metamaterials. Recent studies have explored sub-wavelength resonators in periodic configurations   \cite{ADH21,FA22,QY18} and multilayered elastic structures \cite{BM10,KZD24,WW24} in the sub-wavelength regime have been investigated.  However, the complex topological nature of these periodic and multilayered systems introduces substantial challenges, requiring advanced layer potential theory and spectral analysis of associated layer operators. Due to these complexities, a comprehensive investigation of quasi-Minnaert resonance in such settings remains an open problem. We will address this in future work.


		\bigskip
		
	\noindent\textbf{Acknowledgment.}
	The work of H. Diao is supported by the National Natural Science Foundation of China  (No. 12371422) and the Fundamental Research Funds for the Central Universities, JLU.  The work of H. Liu is supported by the Hong Kong RGC General Research Funds (projects 11304224, 12301420 and 11300821),  the NSFC/RGC Joint Research Fund (project N\_CityU101/21), the France-Hong Kong ANR/RGC Joint Research Grant, A-CityU203/19.

	\end{document}